\documentclass[10pt,a4paper]{amsart}

\usepackage{amsmath}
\usepackage{amsthm}
\usepackage{amsfonts}
\usepackage{amssymb}
\usepackage{mathabx}

\usepackage{pgf,tikz}
\usetikzlibrary{patterns}
\usepackage{pgfplots}

\usepackage{mathrsfs}
\usetikzlibrary{arrows}
\usepackage[mathscr]{euscript}

\usepackage{graphics}
\usepackage{epsfig}
\usepackage{color}
\usepackage{verbatim}

\usepackage{float}
\usetikzlibrary{calc}
\usetikzlibrary{intersections}

\usepackage[utf8]{inputenc}
\usepackage{enumerate}
\numberwithin{equation}{section}
\usepackage[bookmarks = true, colorlinks=true, linkcolor = black, citecolor = black, menucolor = black, urlcolor = black]{hyperref}

\newcommand\R{{\mathbb{R}}}

\newcommand\C{{\mathbb{C}}}

\newcommand\bS{{\mathbb{S}}}
\newcommand\bH{{\mathbb{bH}}}

\renewcommand\Im{{\operatorname{Im}}}

\newcommand\supp{\operatorname{supp}}

\newcommand{\ud}{\mathrm{d}}

\newcommand\bP{{\mathbb{P}}}

\newcommand\CF{{\mathcal F}}
\newcommand\CG{{\mathcal G}}

\newcommand\CR{{\mathcal R}}

\pgfmathsetmacro\sprayRadius{.2pt}
\pgfmathsetmacro\sprayPeriod{.5cm}
\def\x{{\bf X}}

\theoremstyle{plain}
  \newtheorem{theorem}{Theorem}[section]

  \newtheorem{lemma}[theorem]{Lemma}
  \newtheorem{corollary}[theorem]{Corollary}

\theoremstyle{definition}

  \newtheorem{remark}[theorem]{Remark}

\pgfkeys{/pgf/declare function={arcsinh(\x) = ln(\x + sqrt(\x^2+1));}}

\date{\today}
\title[$k$-plane transform and Fourier extension operators]{Bilinear identities involving the $k$-plane transform and Fourier extension operators}
\author{David Beltran and Luis Vega}

\address{David Beltran: Department of Mathematics, University of Wisconsin, 480 Lincoln Drive, Madison, WI, 53706, USA.}
\email{dbeltran@math.wisc.edu}

\address{Luis Vega: Departamento de Matematicas, Universidad del Pais Vasco/Euskal Herriko Unibertsitatea (UPV/EHU), Aptdo. 644, 48080, Bilbao, Spain, and Basque Center for Applied Mathematics (BCAM), Alameda de Mazarredo 14, 48009, Bilbao, Spain}
\email{lvega@bcamath.org, luis.vega@ehu.es}

\thanks{The authors were supported by ERCEA Advanced Grant 2014 669689 - HADE, by the MINECO project MTM2014-53850-P, by Basque Government project IT-641-13 and also by the Basque Government through the BERC 2018-2021 program and by Spanish Ministry of Economy and Competitiveness MINECO: BCAM Severo Ochoa excellence accreditation SEV-2017-0718.}


\begin{document}

\maketitle

\begin{abstract}
We prove certain $L^2(\R^n)$ bilinear estimates for Fourier extension operators associated to spheres and hyperboloids under the action of the $k$-plane transform. As the estimates are $L^2$-based, they follow from bilinear identities: in particular, these are the analogues of a known identity for paraboloids, and may be seen as higher-dimensional versions of the classical $L^2(\R^2)$-bilinear identity for Fourier extension operators associated to curves in $\R^2$.
\end{abstract}


\section{Introduction}


For $n \geq 2$, let $U$ be an open subset in $\R^{n-1}$ and $\phi:\R^{n-1} \to \R$ be a smooth function parametrising a hypersurface $S=\{(\xi, \phi(\xi)) : \xi \in U\}$. Associated to $S$, define the Fourier extension operator
$$
Ef(z):=\int_{U} e^{i (x \cdot \xi + t \phi(\xi) )} f(\xi) \, \ud \xi,
$$
where $z=(x,t) \in \R^{n-1} \times \R$ and $f \in L^1(U)$. The terminology \textit{extension} comes from the fact that $E$ is the adjoint operator to the restriction of the Fourier transform to $S$, that is $E^*h(\xi)=\widehat{h}(\xi, \phi(\xi))$. Stein observed in the late 1960s that under certain curvature hypothesis on $S$ it is possible to obtain $L^p(U)-L^q(\R^{n})$ estimates for $E$ besides the trivial $L^1(U)-L^\infty(\R^n)$ estimate implied by Minkowski's inequality. In particular, the \textit{Fourier restriction conjecture} asserts that if $S$ is compact and has everywhere non-vanishing Gaussian curvature
$$
\| Ef \|_{L^q(\R^n)} \leq C\| f \|_{L^p(U)}
$$
should hold for all $q > \frac{2n}{n-1}$ and $\frac{1}{q} \leq \frac{n-1}{n+1} \frac{1}{p'}$. This conjecture is fully solved for $n=2$ \cite{Fe70, Zygmund74}, but is still open for $n \geq 3$ and constitutes one of the main open problems in Euclidean Harmonic Analysis. The first fundamental result in this direction was the Stein--Tomas \cite{Tomas, Strichartz77} restriction estimate
\begin{equation}\label{eq:SteinTomas}
\| Ef \|_{L^{\frac{2(n+1)}{n-1}}(\R^n)} \leq C \| f \|_{L^2(U)};
\end{equation}
note that this estimate is best possible in terms of the exponent $q$ for $f \in L^2(U)$. Over the last few years, there has been a great interest in establishing the sharp value of $C$ and the existence and characterisation of extremisers in \eqref{eq:SteinTomas} depending on the underlying surface $S$: see, for instance, \cite{FoschiPara} or the most recent survey \cite{FO}.

Substantial improvements on \eqref{eq:SteinTomas} have been achieved over the last few decades. An important ingredient for this has been the bilinear and multilinear approach. Multilinear restriction estimates generally adopt the form
\begin{equation}\label{eq:multilinear}
\| \prod_{j=1}^k E_j f_j \|_{L^{q/k}(\R^n)} \leq C \prod_{j=1}^k \| f_j \|_{L^p(U_j)},
\end{equation}
where the $E_j$ are associated to hypersurfaces $S_j$ satisfying certain \textit{transversality} hypotheses. A key feature of these inequalities is that, under such additional hypotheses, it is possible to obtain estimates for $p=2$ and $\frac{2n}{n-1}<q<\frac{2(n+1)}{n-1}$. The interested reader is referred, for instance, to \cite{W, TaoB} for the theory of bilinear restriction estimates and to \cite{BCT} for the multilinear approach; see also the survey papers \cite{TaoSurvey, BennettSurvey}.

An elementary instance of a bilinear estimate is in fact the identity
\begin{equation}\label{eq:elementary 1D}
\| E_1 f_1 E_2 f_2 \|_{L^2(\R^2)}^2 = (2 \pi)^2 \int_{U_1 \times U_2} \frac{|f_1(\xi_1)|^2 |f_2(\xi_2)|^2}{|\phi_1'(\xi_1)-\phi_2'(\xi_2)|} \, \ud \xi_1 \, \ud \xi_2,
\end{equation}
which follows from an application of Plancherel's theorem and a change of variables; note that under the \textit{transversality} hypothesis $|\phi_1'(\xi_1) - \phi_2'(\xi_2)|>c>0$ for $\xi_1 \in U_1$, $\xi_2 \in U_2$, one may interpret the identity \eqref{eq:elementary 1D} in the framework of \eqref{eq:multilinear}. Of course, the presence of $L^2$ on the left-hand side in \eqref{eq:elementary 1D} is key for the use of Plancherel's theorem. This bilinear approach has its roots in the work of Fefferman \cite{Fe70} and may also be extended to higher dimensions. Identifying $E_j f_j = \widehat{g_j \ud \mu_j}$,\footnote{See $\S$\ref{subsec:FT} for the non-standard normalisation chosen for the Fourier transform $\widehat{ \, \cdot \,}$.} where $g_j:\R^n \to \R$ is the lift of $f_j$ to $S_j$, i.e., $g_j(\xi, \phi_j(\xi))=f_j(\xi)$ and $\ud \mu_j$ is the parametrised measure in $S_j$ defined via 
$$\int_{\R^n} g(\eta) \, \ud \mu(\eta) = \int_{U_j} g(\xi, \phi_j(\xi))\, \ud \xi,$$
one may obtain the $L^2(\R^n)$ bilinear estimate
\begin{equation}\label{eq:elementary hD}
\|E_1 f_1 E_2 f_2 \|_{L^2(\R^n)}^2   \leq \| |g_1|^2 \ud \mu_1 \ast |g_2|^2 \ud \mu_2 \|_{L^1(\R^n)} \| \ud \mu_1 \ast \ud \mu_2 \|_{L^\infty (\R^n)}  \leq C\| f_1 \|_{L^2(U_1)}^2 \| f_2 \|_{L^2{(U_2)}}^2
\end{equation}
after an application of Plancherel's theorem and the Cauchy--Schwarz inequality, provided one assumes the transversality condition $\| \ud \mu_1 \ast \ud \mu_2 \|_{L^\infty (\R^n)} \leq C < \infty$. It should be remarked that the Lebesgue exponent $2$ on the left-hand side of \eqref{eq:elementary hD} corresponds to $q=4\geq \frac{2(n+1)}{n-1}$ if $n \geq 3$; note that  in a bilinear formulation the Lebesgue exponent is interpreted as $q/2$. This is very much in contrast to the setting described in \eqref{eq:multilinear}, in which the main goal is to obtain estimates when $q<\frac{2(n+1)}{n-1}$; bilinear and multilinear estimates of that type are deep and difficult and will not be explored in this paper.

It is interesting to compare \eqref{eq:elementary 1D} and \eqref{eq:elementary hD}. The first observation is that \eqref{eq:elementary 1D} is an identity, whilst \eqref{eq:elementary hD} is an inequality. The second is the presence of the weight factor $|\phi_1'(\xi_1) - \phi_2'(\xi_2)|^{-1}$ in \eqref{eq:elementary 1D}; the transversality weight $|\ud \mu_1 \ast \ud \mu_2|$ in \eqref{eq:elementary hD} does not necessarily have a closed form in terms of the variables of integration of $f_1$ and $f_2$.

The main purpose of this paper is to further exploit the elementary 2-dimensional analysis in \eqref{eq:elementary 1D} in a higher dimensional setting. More precisely, we wish to obtain a bilinear identity in higher dimensions which incorporates an explicit weight factor amounting to some \textit{transversality} condition; we note that an alternative higher dimensional version of \eqref{eq:elementary 1D} has recently been obtained by Bennett and Iliopoulou \cite{BI} in a $n$-linear level. In our goal of obtaining bilinear identities, we shall replace the $L^2(\R^n)$ in \eqref{eq:elementary hD} by a mixed-norm $L^1(\R^{n-2}) \times L^2(\R^2)$. Given $x=(\bar x, x'') \in \R^{n-2} \times \R^2$, taking the $L^1$-norm in the $\bar x$ variables will essentially reduce matters to a 2-dimensional analysis in the $x''=(x_{n-1},x_n)$ variables, where the resulting extension operators $E_1$ and $E_2$ will correspond to sections of the original surfaces by 2-dimensional planes parallel to $\xi_1=\dots = \xi_{n-2}=0$. The existence of such bilinear identities has already been established by Planchon and the second author \cite{PV} if the underlying hypersurfaces are paraboloids. The motivation in their work came from the relevant role played by these types of inequalities in the global behaviour of large solutions of nonlinear Schr\"odinger equations; see the next subsection for further details. 
Here we further explore whether bilinear identities hold for two other fundamental surfaces: the sphere and the hyperboloid.

Before describing our results in detail we shall first review the known results in the case of paraboloids, as they will provide the framework and context to understand our results.


\subsection{Estimates for paraboloids and connections to Schrödinger equations}\label{subsec:paraboloids}

In recent years, starting with the work of Ozawa and Tsutsumi \cite{OT} for the paraboloid $S_1=S_2=\{(\xi, |\xi|^2) : \xi \in  \R^{n-1}\}$, there has been an increasing interest in understanding the weight $|\ud \mu_1 \ast \ud \mu_2|$ in \eqref{eq:elementary hD} so that a $L^2$-bilinear estimate
\begin{equation}\label{eq:general bilinear}
\| E_1 f_1 E_2 f_2 \|_{L^2(\R^n)}^2 \leq C \int_{U_1 \times U_2} K_{S_1,S_2}(\xi_1,\xi_2) |f_1(\xi_1)|^2 |f_2(\xi_2)|^2 \, \ud \xi_1 \, \ud \xi_2
\end{equation}
holds for some kernel $K_{S_1,S_2}$ and such that the constant $C$ is best possible; in many cases, extremisers for the above kinds of inequalities have also been characterised. This has been mostly studied for paraboloids \cite{Carneiro2009}, cones \cite{BezRogers}, spheres \cite{FK, CO} and hyperboloids \cite{OR,Jeavons}, with the corresponding natural interpretations in PDE.

It should be noted that the bilinear estimates \eqref{eq:elementary 1D} and \eqref{eq:elementary hD} also hold when $E_2 f_2$ is replaced by its complex conjugate $\overline{E_2 f_2}$. This is, of course, of interest when $S_1=S_2$ and $f_1=f_2$, as then the bilinear estimates can be reinterpreted as $L^2$ estimates of $|Ef|^2$. In particular, in the case of paraboloids, the identity \eqref{eq:elementary 1D} may be reinterpreted as
\begin{equation}\label{OT 1}
\int_{\R \times \R} |D_{x} |u|^2|^2 \, \ud x \, \ud t = \frac{1}{2(2\pi)^2} \int_{\R \times \R} |\xi-\eta| |\widehat{u_0}(\xi)|^2 |\widehat{u_0}(\eta)|^2 \, \ud \xi \, \ud \eta
\end{equation}
or simply
\begin{equation}\label{OT 1 bis}
\int_{\R \times \R} |D_{x}^{1/2} |u|^2|^2 \, \ud x \, \ud t = \frac{1}{2} \| u_0 \|_{L^2(\R)}^2 \| u_0 \|_{L^2(\R)}^2 
\end{equation}
in order to avoid the singularity of the resulting weight $|\phi'(\xi)-\phi'(\eta)|=2|\xi-\eta|$; here we interpret the extension operator $u(x,t)=  E\widecheck{u_0} (x,t)$ as the solution of the free Schrödinger equation $i \partial_t u - \Delta u =0$ in $\R^d$ associated to the initial data $u(x,0)=u_0(x)$, with the normalisation of the Fourier transform considered in $\S$\ref{subsec:FT}. Note that, for this specific case, it is crucial that the multiplier associated to $D_{x}$ coincides precisely with $|\phi'(\xi)-\phi'(\eta)|$. Moreover, Ozawa and Tsutsumi \cite{OT} made use of the Radon transform to obtain the higher dimensional version
\begin{equation}\label{OT d}
\| (-\Delta)^{(2-d)/4} |u|^2 \|_{L^2_{x,t}(\R^d \times \R)}^2 \leq \textbf{OT}(d) \| u_0 \|_{L^2(\R^d)}^2 \| u_0 \|_{L^2(\R^d)}^2 
\end{equation}
where the constant $\textbf{OT}(d)=\frac{2^{-d}\pi^{(2-d)/2}}{\Gamma(d/2)}$ is sharp after verifying that for $u_0(x)=e^{-|x|^2}$ the inequality becomes an identity; see also \cite{Carneiro2009, BBJP}. 

The interest of Ozawa and Tsutsumi comes from the noninear Schr\"odinger equation
\begin{equation}\label{eq:nonlinear}
i\partial_tu+\partial^2 u=i\lambda(\partial|u|^2)u + f(u),
\end{equation}
where $\lambda\in\R$ and $f$ is a nonlinear interaction, which can be taken to be zero for simplicity of this exposition. In \cite{OT} the authors proved a well-posedness result in the Sobolev space $H^{1/2}(\R)$. This was a non-trivial task due to the presence of the derivative term $\partial|u|^2$ on the right-hand side of the equation \eqref{eq:nonlinear}. The advantage of \eqref{OT 1 bis} (or \eqref{OT d} when $d=1$) as opposed to an inequality of the type \eqref{eq:general bilinear} is the gain in derivatives of the solution with respect to the intial data, which allowed the authors to treat the term $\partial |u|^2$ as a perturbation. 

The result by Ozawa and Tsutsumi was not further explored until \cite{PV}, where Planchon and the second author  established certain higher dimensional analogues of the $\R^{1+1}$ identity \eqref{OT 1}. Their identities also involved the Radon transform in the spatial variables\footnote{Note that the Radon transform in the spatial variables in $u(x,t)$ amounts to a $(n-2)$-plane transform in the context of the extension operators $Ef(z)$.}, which in fact features in the statement. Recall that given a linear $k$-dimensional subspace $\pi \in \mathcal{G}_{k,n}$ and $y \in \pi^\perp$,  the $k$-plane transform of a function $f$ belonging to a suitable a \textit{priori} class is defined as
$$
T_{k,n} f (\pi,y):=\int_{\pi} f(x+y)\, \ud \lambda_{\pi} (x),
$$
where $\mathcal{G}_{k,n}$ denotes the Grassmanian manifold of all $k$-dimensional subspaces in $\R^n$ and $\ud \lambda_{\pi}$ is the induced Lebesgue measure on $\pi$. The cases $k=1$ and $k=n-1$ correspond to the X-ray transform $X$ and the Radon transform\footnote{The Radon transform $\mathcal{R}f$ is identified with a function in $\bS^{n-1}_+ \times \R$ setting $\mathcal{R}f (\omega, s) \equiv \mathcal{R}f (\langle \omega \rangle^\perp, s \omega)$.} $\mathcal{R}$ respectively. With this notation, the following was shown in \cite{PV}. 
\begin{theorem}[\cite{PV}]\label{PV thm}
Let $n \geq 2$ and $\omega \in \bS^{d-1}$. Then,
\begin{equation}\label{PV d}
\int_{\R}\! \int_{\R} |\partial_s \mathcal{R} (|u (\cdot, t)|^2) (\omega, s)|^2 \, \ud s \, \ud t + J_\omega(u) \! =\! \frac{\pi}{(2\pi)^{2d+1}} \int_{\R^d} \! \int_{\R^d} |(\xi-\eta) \cdot \omega| |\widehat{u_0}(\xi)|^2 |\widehat{u_0}(\eta)|^2 \, \ud \xi \, \ud \eta,
\end{equation}
where
\begin{align*}
J_\omega(u)\! :=\! \int_{\R}\! \int_{\R} \! \int_{(\langle \omega \rangle^\perp)^2} \! \big|u(x + s\omega, t) \partial_{s} u (y + s\omega, t)-  u(y + s\omega, t) \partial_{s} u (x + s\omega, t)\big|^2 \, \ud \lambda_{(\langle \omega \rangle^\perp)^2}(x,y)\, \ud s  \,  \ud t.
\end{align*}
\end{theorem}

Note that fixing $\omega=e_d$ (or any other coordinate vector) in \eqref{PV d} above, the first term on the left-hand side amounts to $\| \partial_s \| |u|^2 \|_{L^1(\R^{d-1})} \|_{L^2_{x_d,t}(\R^2)}^2$, which in the absence of the derivative $\partial_s$ becomes $\| u \|^4_{L^4_{x_d,t}(\R^2; L^2(\R^{d-1}))}$; note the contrast with the $L^4$-nature of \eqref{eq:elementary 1D} and \eqref{eq:elementary hD}.

The approach used in \cite{PV} to establish Theorem \ref{PV thm} uses integration-by-parts arguments  
and extends to versions of \eqref{PV d} for nonlinear Schrödinger equations with nonlinearity of the type $\pm |u|^{p-1}u$, where $p \geq 1$. The motivation in \cite{PV} is similar to that of Ozawa and Tsutsumi, and comes, more precisely, from the breakthrough result by Colliander, Keel, Staffilani, Takaoka, and  Tao in \cite{Iteam}, who established global well-posedness of the critical defocusing $3d$ nonlinear Schrödinger equation (NLS)
$$
i  \partial_t u + \Delta u = u |u|^4
$$
in the energy space. This builds up on a previous result of Bourgain  \cite{Bourgain1999}, who showed the well-posedness under the assumption of radial symmetry.
Bourgain used an ad hoc modification of a well-known weighted estimate, typically referred to as Morawetz inequality, proved in \cite{LS}; see also \cite{Morawetz}. The weights are of the type $|x|^{-1}$ and therefore not translation invariant, leading to the well-posedness only under the radial assumption. To overcome that obstacle, the authors in \cite{Iteam} established a bilinear Morawetz estimate that avoids the loss of the translation symmetry. Whilst their strategy was successful in dimension 3, the method has some obstructions when considering nonlinear Schrödinger equations in dimensions 1 and 2. The nonlinear versions of the identities \eqref{PV d} were then used in \cite{PV} to prove certain lower dimensional well-posedness results, also obtained  independently by Colliander, Grillakis and Tzirakis \cite{CGT} by different methods. It is remarked that \eqref{PV d} has further applications, such as well-posedness in $3d$ for exterior domains, scattering of solutions (see also \cite{Nakanishi}) or recovering Bourgain's \cite{Bourgain98} bilinear refinement of the Strichartz estimate; the reader is referred to \cite{PV} for further details.

Unfortunately, the bilinear identities \eqref{PV d} (or more precisely, the integration-by-parts proof method) are extremely rigid and they rely on the fact that the Schr\"odinger equation is a system with a quadratic dispersion relation. However, the connections of these estimates with Strichartz inequalities suggest that similar identities should also be true for general dispersion relations. This is what we start to explore in this paper for the particular case of the Helmholtz and Klein-Gordon equations. Our approach completely relies on Fourier Analysis techniques, after noting that \eqref{PV d} can be obtained from applications of Plancherel's theorem, in the spirit of \eqref{eq:elementary 1D} and \eqref{eq:elementary hD}. Of course, such a proof method only applies to linear problems, and it is therefore more natural to understand our results in the context of the interaction of the $k$-plane transform and $|Ef|^2$, where the Fourier extension operator $E$ is associated to spheres and hyperboloids; see also the recent paper \cite{BBFGI} or the upcoming preprint \cite{BN} for further examples of this interaction. Despite the lack of nonlinear results, we expect that the identification of bilinear identities for the linear Helmholtz and Klein--Gordon equations presented in this article will provide some light to develop methods based on direct integration by parts, which would be more amenable to nonlinear counterparts. This will be explored somewhere else in the future.




\subsection{Estimates for the sphere}
In the case of the sphere $\bS^{n-1}_r \equiv r \bS^{n-1}$ of radius $r$ in $\R^n$, consider the more classical form of the extension operator
$$
g \mapsto \widehat{g \ud \sigma^n_r},
$$
where $\ud \sigma^n_r$ denotes the induced normalised Lebesgue measure on $\bS^{n-1}_r$ and $g \in L^1(\bS^{n-1}_r)$. The following $L^2$-identities for $T_{n-2,n}(\widehat{g_1\ud \sigma^n_r} \overline{\widehat{g_2\ud \sigma^n_r}} )$ are obtained.
\begin{theorem}\label{thm:sphere1}
Let $n \geq 3$. Let $\pi \in \mathcal{G}_{n-2,n}$ and let $\pi^\perp$ denote the orthogonal subspace to $\pi$. For each $z \in \R^n$, write $z=z^\pi + z^\perp$, where $z^\pi$ is the orthogonal projection of $z$ into $\pi$. Then
\begin{align}\label{id:sph}
\int_{ \pi^\perp} \Big|  (-\Delta_y & )^{1/4} T_{n-2,n}  (\widehat{g_1\ud \sigma^n_r} \overline{\widehat{g_2\ud \sigma^n_r}}) (\pi,y) \Big|^2 \, \ud \lambda_{\pi^\perp}(y)  \\
& = \mathbf{C}_{\bS^{n-1}}  \int_{ (\bS^{n-1}_r)^2 } K_{\pi, \bS^{n-1}_r} (\xi,\zeta)g_1(\xi) \bar{g}_2(\xi^\pi + \tilde{\xi}^\perp) g_2(\zeta) \bar{g}_1(\zeta^\pi + \tilde{\zeta}^\perp)  \, \ud \sigma^n_r (\xi) \, \ud \sigma^n_r (\zeta) \notag
\end{align}
where
$$
K_{\pi, \bS^{n-1}_r}(\xi, \zeta):= \frac{2}{|\xi^\perp + \zeta^\perp|}, \qquad  \quad \mathbf{C}_{\bS^{n-1}}:=(2\pi)^{2(n-1)},
$$
$r_{\xi}=\sqrt{r^2- |\xi^\pi|^2}$ and $\tilde{\xi}^\perp, \tilde{\zeta}^\perp \in \pi^\perp$ are the reflected points of $\xi^\perp$ and $\zeta^\perp$ in $\pi^\perp$ with respect to the line passing through the origin and $\xi^\perp + \zeta^\perp$, that is $\xi^\perp + \zeta^\perp = \tilde{\xi}^\perp + \tilde{\zeta}^\perp$ (see Figure \ref{figure:circles new points}).
\end{theorem}

\tikzset{%
    add/.style args={#1 and #2}{
        to path={%
 ($(\tikztostart)!-#1!(\tikztotarget)$)--($(\tikztotarget)!-#2!(\tikztostart)$)%
  \tikztonodes},add/.default={.2 and .2}}
}  

\begin{figure}
\begin{center}

\begin{tikzpicture}

\draw[name path= circle1, blue] (0,0) circle (1cm);

\draw[name path= circle2, red] (0,0) circle (2cm);

\coordinate (aux1) at ({1*sin(25) + 2*sin(75)}, {1*cos(25)+2*cos(75)});

\coordinate (aux2) at ({2*sin(75)}, {2*cos(75)});

\coordinate (aux3) at ({1*sin(25)}, {1*cos(25)});

\coordinate (O) at (0,0);

 \draw[black, ->] (O) -- (aux3);
 
 \node[above] at (0.6, 0.9) {$\xi^\perp$}; 
 
  \draw[black, ->] (O) -- (aux2) node [above, right] {$\zeta^\perp$};


\path [name path=line1, add= 2 and 1, blue] ($(O)!(aux3)!(aux1)$) to (aux3);

\path [name intersections={of=circle1 and line1, by={x,x'}}];


\draw [dashed, black] (x) to (x');


\path [name path=line2, add=2 and 1,  blue] ($(O)!(aux2)!(aux1)$) to (aux2);

\draw[dashed, olive, ->] (0,0) -- (aux1) node [above, right] {$\xi^\perp + \zeta^\perp$}; ;

\path [name intersections={of=circle2 and line2, by={y,y'}}];


\draw[black, dashed] (y) to (y');

\draw[black, ->] (0,0) -- (x') node [below, right] {$\tilde{\xi}^\perp$}; 

\draw[black, ->] (0,0) -- (y) node [above, right] {$\tilde{\zeta}^\perp$};

 \node[below] at  (0,0) {$0$};
 
\node[red,below, scale = 1] at  (2,-1) {$r_\zeta^\pi \bS^1$};
\node[blue,below, scale = 1] at  (1,-0.7) {$r_\xi^\pi  \bS^1$};

\end{tikzpicture}
\caption{The new points $\tilde{\xi}^\perp \in r_\xi^\pi \bS^1$ and $\tilde{\zeta}^\perp \in r_\zeta^\pi \bS^1$ in $\pi^\perp$ are the reflected points of $\xi^\perp$ and $\zeta^\perp$ with respect to $\xi^\perp + \zeta^\perp$.}
\label{figure:circles new points}
\end{center}
\end{figure}
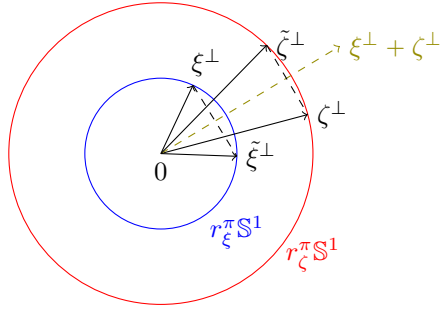

Of course, the $L^2$-nature of the inequality on its left-hand side allows one to take advantage of Plancherel's theorem.  As briefly described before $\S$\ref{subsec:paraboloids}, the key presence of the $(n-2)$-plane transform reduces the problem to a 2-dimensional analysis, and one is left to understand the convolution of two weighted measures associated to concentric circles of different radii in the subspace $\pi^\perp \simeq \R^2$. The main advantage with respect to \eqref{eq:elementary hD} is that in this setting it is possible to express $ h_1 \ud \sigma_{r_1}^2 \ast h_2 \ud \sigma_{r_2}^2 (\xi^\perp + \zeta^\perp)$ as the weight $\ud \sigma_{r_1}^2 \ast \ud \sigma_{r_2}^2(\xi^\perp + \zeta^\perp)$ times an evaluation of the functions $h_1$ and $h_2$ at points depending on $\xi^\perp$ and $\zeta^\perp$. 

Several interesting corollaries can be deduced from Theorem \ref{thm:sphere1}; their short proofs will be given in $\S$\ref{sec:corollaries}. Given complex numbers $a, b, c, d \in \C$, the well known identity
\begin{equation}\label{eq:complex numbers}
a \bar b \bar c d = \frac{1}{2} \big( |ac|^2 + |bd|^2 - | a \bar c - b \bar d|^2 \big) + i \, \Im (a \bar b \bar c d)
\end{equation}
may be used in Theorem \ref{thm:sphere1} to replace the 4-linear \textit{wave interaction}
$$
g_1(\xi) \bar{g}_2(\xi^\pi + \tilde{\xi}^\perp) g_2(\zeta) \bar{g}_1(\zeta^\pi + \tilde{\zeta}^\perp)
$$
in \eqref{id:sph} by an alternative expression involving $|g_1(\xi)|^2 |g_2(\zeta)|^2$ and which is closer in spirit to \eqref{PV d}. 

\begin{corollary}\label{cor:sphere1early}
Let $n \geq 3$ and $\pi \in \mathcal{G}_{n-2,n}$. Then
\begin{align}\label{eq:id sphere minus}
\int_{ \pi^\perp} \Big| (-\Delta_y )^{1/4} & T_{n-2,n}  (\widehat{g_1\ud \sigma^n_r} \overline{\widehat{g_2\ud \sigma^n_r}}) (\pi,y) \Big|^2 \, \ud y   \\
& = \mathbf{C}_{\bS^{n-1}}  \int_{ (\bS^{n-1}_r)^2 } \!\!\! K_{\pi, \bS^{n-1}_r} (\xi,\zeta) |g_1 ( \xi )|^2 |g_2 (\zeta)|^2  \, \ud \sigma^n_r (\xi) \, \ud \sigma^n _r(\zeta) - I_{\pi, \bS^{n-1}_r}(g_1,g_2), \notag
\end{align}
where
$$
I_{\pi, \bS^{n-1}_r}(g_1,g_2):=\frac{ \mathbf{C}_{\bS^{n-1}}}{2} \int_{ (\bS^{n-1}_r)^2 } \!\!\! K_{\pi, \bS^{n-1}_r} (\xi,\zeta) | g_1 ( \xi ) g_2(\zeta) -  g_2 (\xi^\pi +  \tilde{\xi}^\perp ) g_1(\zeta^\pi + \tilde{\zeta}^\perp) |^2   \, \ud \sigma^n_r (\xi) \, \ud \sigma^n_r (\zeta).
$$
\end{corollary}
Of course, the term $I_{\pi, \bS^{n-1}_r}(g_1,g_2) \geq 0$ and is identically zero if $g_1$ and $g_2$ are constant functions, so it may be dropped from \eqref{eq:id sphere minus} at the expense of losing the identity, leading to a sharp inequality which fits in the context of \eqref{eq:general bilinear}. Thus, the term $I_{\pi, \bS^{n-1}_r}(g_1,g_2)$ may be interpreted as the \textit{distance} of such a resulting inequality to become an identity.\footnote{The inequality resulting from dropping $I_{\pi, \bS^{n-1}_r}(g_1,g_2)$ in \eqref{eq:id sphere minus} may be obtained more directly by an application of the Cauchy--Schwarz inequality: see $\S$\ref{subsec:cor sphere 1}}

As the $k$-plane transform satisfies the Fourier transform relation
\begin{equation}\label{eq:FT and kplane}
\mathcal{F}_y T_{k,n} f (\pi, \xi) = \widehat{f}(\xi) \qquad \textrm{ for $\xi \in \pi^\perp$},
\end{equation}
one may easily obtain by means of Plancherel's theorem the relation
\begin{equation}\label{eq:plancherel k-plane}
 \| f \|_{L^2(\R^n)}^2 = \frac{(2\pi)^{-k}}{|\mathcal{G}_{n-k-1, n-1}|} \| (-\Delta_y)^{k/4} T_{k,n} f \|_{L^2(\mathcal{G}_{k,n}, L^2(\pi^\perp))}^2;
\end{equation}
see $\S$\ref{sec:preliminaries} for further details. Thus, on averaging Theorem \ref{thm:sphere1} over all $\pi \in \mathcal{G}_{n-2,n}$ one has the following.
\begin{corollary}\label{cor:sph2}
Let $n \geq 3$. Then
\begin{equation*}
\| (-\Delta)^{\frac{3-n}{4}} ( \widehat{g_1\ud \sigma^n_r} \overline{\widehat{g_2\ud \sigma^n_r}}) \|^2_{L^2(\R^n)} \leq (2\pi)^{2-n} \mathbf{C}_{\bS^{n-1}}  \int_{(\bS^{n-1}_r)^2} K_{\bS^{n-1}_r}(\xi,\zeta)  |g_1 ( \xi )|^2 |g_2 (\zeta)|^2  \, \ud \sigma^n_r (\xi) \, \ud \sigma^n_r (\zeta)
\end{equation*}
where
$$
K_{\bS^{n-1}_r}(\xi,\zeta):=\frac{1}{|\mathcal{G}_{1,n-1}|}\int_{\CG_{n-2, n}} K_{\pi, \bS^{n-1}_r}(\xi,\zeta) \, \ud \mu_{\mathcal{G}}(\pi).
$$
\end{corollary}
In the particular case $n=3$ and after setting $g_1=g_2$, the right-hand side in Corollary \ref{cor:sph2} amounts to a bilinear quantity appearing in the work of Foschi \cite{Foschi} on the sharp constant in the Stein--Tomas inequality \eqref{eq:SteinTomas} for $\bS^2$. Thus, appealing to his work, one can deduce the following.

\begin{corollary}[Stein--Tomas \cite{Tomas}, Foschi \cite{Foschi}]\label{cor:sphST}
\begin{equation}\label{eq:ST sphere}
\| \widehat{g \ud \sigma^3} \|_{L^4(\R^3)} \leq 2\pi \| g \|_{L^2(\bS^{2})}.
\end{equation}
\end{corollary}
Besides the value for the sharp constant, Foschi \cite{Foschi} also showed that the only real valued extremisers are constant functions; the existence of extremisers was previously verified in \cite{ChristShao1, ChristShao2}.

\subsubsection*{Solution to the Helmholtz equation}

Consider the Helmholtz equation $\Delta u + k^2 u=0$ in $\R^n$. If $\sup_{R>0} \frac{1}{R} \int_{B_R} |u|^2 < \infty$, then there exists $g \in L^2(\bS^{n-1}_k)$ such that $u=\widehat{g \ud \sigma_{\bS^{n-1}_k}}$. Theorem \ref{thm:sphere1} and the subsequent corollaries may be then interpreted in that context.


%
\subsection{Estimates for the  hyperboloid}

A similar analysis to the one described for $\bS^{n-1}$ may be carried for one of the components of the elliptic hyperboloid in $\R^{d+1}$, defined by
$$
\mathbb{H}^d_m := \{ (\xi, \xi_{d+1}) \in \R^d \times \R : \xi_{d+1} = \phi_m(\xi):= \sqrt{m^2+|\xi|^2}\}
$$
and equipped with the Lorentz invariant measure $\ud \sigma_{\mathbb{H}_m^d}$ (see $\S$\ref{subsec:Lorentz}), defined by
\begin{equation}\label{eq:Lorentz inv measure}
\int_{\mathbb{H}_m^d} g(\xi,\xi_{d+1}) \, \ud \sigma_{\mathbb{H}_m^d}(\xi,\xi_{d+1}) = \int_{\R^d} g(\xi, \phi_m(\xi)) \frac{\ud \xi}{\phi_m(\xi)}.
\end{equation}
A function $f \in L^1(\R^d)$ is identified with its lift $g$ to $\mathbb{H}_m^d$, given by $g(\xi, \phi_m(\xi)) = f(\xi)$, and note
$$
\widehat{g \ud \sigma_{\mathbb{H}_m^d}} (x,t) = \int_{\R^d} e^{i x \cdot \xi } e^{i t \sqrt{m^2+|\xi|^2}}f(\xi) \frac{\ud \xi}{\sqrt{m^2+|\xi|^2}}
$$
where $(x,t) \in \R^n = \R^d \times \R$. A natural reason to split into a space-time domain is in view of the connection of $\widehat{g \ud \sigma_{\mathbb{H}^d_m}}$ with the Klein--Gordon propagator $e^{i t \sqrt{m^2 - \Delta}}f$; this will be further discussed below. Thus, considering the Radon transform in the space variables - as in \eqref{PV d} and as opposed to Theorem \ref{thm:sphere1}, where no time role is given and therefore $(n-2)$-plane transform is taken - one obtains the following (see $\S$\ref{subsec:Lorentz} for the definition of Lorentz transformation).

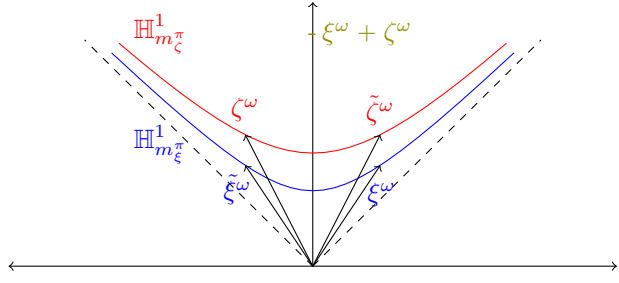
\begin{figure}

\begin{tikzpicture}

    \draw [<->]  (-4,0) to (4,0) ;

\draw [->] (0,0) to (0,3.5);        
        
        \draw [-, dashed] (0,0) to (3,3);
        \draw [-, dashed] (0,0) to (-3,3);

    \pgfmathsetmacro{\a}{1}
    \draw[name path=hyp1 , blue] plot[domain=-1.7:1.7] ({\a*sinh(\x)},{\a*cosh(\x)});
    
\node[red, above] at (-2,2.7) {$\bH_{m_\zeta^\pi}^1$};     
    
    \pgfmathsetmacro{\as}{1.5}    
    \draw[red] plot[domain=-1.3:1.3] ({\as*sinh(\x)},{\as*cosh(\x)});

\node[blue, below] at (-2,2) {$\bH_{m_\xi^\pi}^1$};

\coordinate (xi) at ({\a*sinh(0.8)},{\a*cosh(0.8)});
\coordinate (zeta) at ({-\a*sinh(0.8)},{\as*cosh(   arcsinh(sinh(0.8)/1.5)    )});

\coordinate (sum) at ($(xi) + (zeta)$);

\draw[->] (0,0) -- (xi) ;
\draw[->] (0,0) -- (zeta);

\coordinate (xi_modified) at ( $(xi) + (0.02,-0.1)$);
\coordinate (zeta_modified) at ( $(zeta) + (0,0.05)$);

\node[below, blue] at (xi_modified) {$\xi^\omega$};
\node[above, red] at (zeta_modified) {$\zeta^\omega$};

\node[olive] at (sum) {-};
\node[right, olive] at (sum) {$\xi^\omega+ \zeta^\omega$};

\coordinate (xi_new) at ({-\a*sinh(0.8)},{\a*cosh(0.8)});
\coordinate (zeta_new) at ({\a*sinh(0.8)},{\as*cosh(   arcsinh(sinh(0.8)/1.5)    )});

\draw[->] (0,0) -- (xi_new) ;
\draw[->] (0,0) -- (zeta_new);

\coordinate (xi_new_modified) at ( $(xi_new) - (0.1,0)$);

\node[below,blue] at (xi_new_modified) {$\tilde{\xi}^\omega$};
\node[above, red] at (zeta_new) {$\tilde{\zeta}^\omega$};

\end{tikzpicture}
\caption{If $\xi^\omega + \zeta^\omega$ lies in the vertical axis, the new points $\tilde{\xi}^\omega$ and $\tilde{\zeta}^\omega$ are the reflected points of $\xi^\omega$ and $\zeta^\omega$ with respect to that axis. For ease of notation, $\xi^\omega$ is identified with the point $(\xi^\omega, \phi_{m_\xi^\pi}(\xi^\omega)) \in \bH^1_{m_\xi^\pi}$, and similarly for the other points. Note that, in this situation, the Lorentz transformation $L$ in Theorem \ref{thm:hyp 1} is the identity. In general, the above situation results after applying $L$, which maps $\xi^\omega + \zeta^\omega$ to the vertical axis.}
\label{figure:hyp new points origin}
\end{figure}

\begin{theorem}\label{thm:hyp 1}
Let $d \geq 2$. Let $\omega \in \bS^{d-1}_+$ and let $\pi:=\langle \omega \rangle^\perp \in \mathcal{G}_{d-1,d}$ be the orthogonal subspace to $\langle \omega \rangle$. For each $x \in \R^d$, write $x=x^\pi + x^\omega \omega$, where $x^\omega= x \cdot \omega$. Then 
\begin{align*}
 \int_{\R} \int_{\R} \Big|  \partial_{s}^{1/2} \mathcal{R}  (  \widehat{g_1  \ud \sigma_{\mathbb{H}_m^d}} & (\cdot, t) \overline{\widehat{g_2  \ud \sigma_{\mathbb{H}_m^d}}}  (\cdot, t) ) (\omega, s)  \Big|^2 \, \ud s \, \ud t  \\
 & = \mathbf{C}_{\mathbb{H}^d}
\int_{(\R^d)^2} K_{\omega, \mathbb{H}_m^d} (\xi, \zeta) f_1(\xi) \bar f_2 (\xi^\pi + \tilde{\xi}^\omega \omega)  f_2(\zeta) \bar f_1 ( \zeta^\pi + \tilde{\zeta}^\omega \omega) \, \frac{\ud \xi}{\phi_m(\xi)} \, \frac{\ud \zeta}{\phi_m(\zeta)} 
\end{align*}
where
$$
K_{\omega, \mathbb{H}_m^d}(\xi, \zeta) := \frac{|\xi^\omega  - \tilde{\xi}^\omega |^{1/2} |\zeta^\omega  - \tilde{\zeta}^\omega |^{1/2}}{|\xi^\omega \phi_m(\zeta) - \zeta^\omega \phi_m(\xi)|} \qquad \text{and} \qquad \mathbf{C}_{\mathbb{H}^d} = (2\pi)^{2d}.
$$
Above, the points $(\tilde{\xi}^\omega, \phi_{m_{\xi}^\pi } (\tilde{\xi}^\omega) ) \in \bH^1_{m_{\xi}^\pi}$ and $(\tilde{\zeta}^\omega, \phi_{m_{\zeta}^\pi } (\tilde{\zeta}^\omega)) \in \bH^1_{m_\zeta^\pi}$ are the image under $L^{-1}$ of the reflected points of $L( (\xi^\omega, \phi_{m_{\xi}^\pi } (\xi^\omega)))$ and $L( (\zeta^\omega, \phi_{m_{\zeta}^\pi } (\zeta^\omega) ))$ in $\R^2$ with respect to the vertical axis respectively, where $L$ is the unique Lorentz transformation mapping $(\xi^\omega + \zeta^\omega , \phi_{m_{\xi}^\pi } (\xi^\omega) + \phi_{m_{\zeta}^\pi } (\zeta^\omega) )$ to the vertical axis and $m_\xi^\pi:=\sqrt{m^2+|\xi^\pi|^2}$ (see Figure \ref{figure:hyp new points origin}).
\end{theorem}

As in the case of the sphere, the use of the Radon transform in $\R^d$ and Plancherel's theorem reduces the above estimate to explicitly understand $h_1 \ud \sigma_{\mathbb{H}^1_{m_1}} \ast h_2 \ud \sigma_{\mathbb{H}^1_{m_2}} (\xi^\omega + \zeta^\omega, \phi_{m_1}(\xi^\omega) + \phi_{m_2}(\zeta^\omega))$. In fact, note that the value of $K_{\omega, \bH^d_m}$ amounts to the expression $
\left| \left( \begin{smallmatrix}
1 & 1 \\
\phi'_{m_\xi^\pi}(\xi^\omega) & \phi'_{m_\zeta^\pi}(\zeta^\omega)
\end{smallmatrix} \right) \right|
$
corrected with the natural weight $\frac{1}{\phi_{m_\xi^\pi}(\xi^\omega) \phi_{m_\zeta^\pi}(\zeta^\omega)}$ coming from the definition of $\ud \sigma_{\bH^1_{m}}$. This should be compared with the elementary two-dimensional identity \eqref{eq:elementary 1D}. The presence of the numerator $|\xi^\omega - \tilde{\xi}^\omega|^{1/2} |\zeta^\omega - \tilde{\zeta}^\omega|^{1/2}$ is due to the action of $\partial_s^{1/2}$ on $\mathcal{R}  (  \widehat{g_1  \ud \sigma_{\mathbb{H}_m^d}}  (\cdot, t) \overline{\widehat{g_2  \ud \sigma_{\mathbb{H}_m^d}}}  (\cdot, t) )$. Moreover, one can explicitly write $\tilde{\xi}^\omega$ and $\tilde{\zeta}^\omega$ in terms of $\xi, \zeta$ and $\omega$, leading to the more compact expression
$$
K_{\omega, \bH^d_m}(\xi, \zeta)=  \frac{2(\phi_m(\xi) + \phi_m(\zeta) )}{(\phi_m(\xi) + \phi_m(\zeta) )^2-( (\xi+\zeta) \cdot \omega )^2}.
$$

As in the case of the sphere, several corollaries can be deduced from Theorem \ref{thm:hyp 1}. As for Corollary \ref{cor:sphere1early}, one may use \eqref{eq:complex numbers} to rewrite Theorem \ref{thm:hyp 1} in the spirit of Theorem \ref{PV thm}.

\begin{corollary}\label{cor:hyp 1}
Let $d \geq 2$ and $\omega \in \bS^{d-1}_+$. Then 
\begin{align*}
 \int_{\R} \! \int_{\R} \Big|  \partial_{s}^{1/2}  \mathcal{R} ( \widehat{g_1  \ud \sigma_{\mathbb{H}_m^d}} & (\cdot, t) \overline{\widehat{g_2  \ud \sigma_{\mathbb{H}_m^d}}}  (\cdot, t) ) (\omega,s)  \Big|^2  \! \! \ud s \, \ud t \!  \\
 & = \!  \mathbf{C}_{\mathbb{H}^d} \! \!
\int_{(\R^d)^2} \!\!\!\! K_{\omega, \mathbb{H}^d_m} (\xi, \zeta) |f_1(\xi)|^2 |f_2(\zeta)|^2  \frac{\ud \xi}{\phi_m(\xi)}  \frac{\ud \zeta}{\phi_m(\zeta)} \! - \! I_{\omega, \bH^d_m}(f_1,f_2)
\end{align*}
where 
$$
I_{\omega, \bH^d_m}(f_1,f_2):= \frac{\mathbf{C}_{\mathbb{H}^d}}{2} \int_{(\R^d)^2} K_{\omega, \mathbb{H}^d_m} (\xi, \zeta) | f_1(\xi) f_2(\zeta) - f_2(\xi^\pi + \tilde{\xi}^\omega \omega ) f_1(\zeta^\pi + \tilde{\zeta}^\omega \omega )|^2 \, \frac{\ud \xi}{\phi_m(\xi)} \, \frac{\ud \zeta}{\phi_m(\zeta)}.
$$
\end{corollary}

As $\mathcal{R}=T_{d-1,d}$, the use of the Plancherel's relation \eqref{eq:plancherel k-plane} after averaging over $\omega \in \bS^{d-1}_+$ yields the following analogue of Corollary \ref{cor:sph2} for hyperboloids.
\begin{corollary}\label{cor:hyp 2}
Let $d \geq 2$. Then
\begin{equation*}
\| (-\Delta_x)^{\frac{2-d}{4}} ( \widehat{g_1  \ud \sigma_{\mathbb{H}_m^d}}  \overline{\widehat{g_2  \ud \sigma_{\mathbb{H}_m^d}}} )  \|_{L^2_{x,t}(\R^d \times \R)}^2 \leq (2\pi)^{1-d} \mathbf{C}_{\mathbb{H}^d}
\int_{(\R^d)^2} \!\!\!\! K_{ \mathbb{H}^d_m} (\xi, \zeta) |f_1(\xi)|^2 |f_2(\zeta)|^2 \, \frac{\ud \xi}{\phi_m(\xi)} \, \frac{\ud \zeta}{\phi_m(\zeta)}
\end{equation*}
where
$$
K_{ \mathbb{H}^d_m} (\xi, \zeta):= \frac{1}{2} \int_{\bS^{d-1}} K_{\omega, \mathbb{H}^d_m} (\xi, \zeta) \, \ud \sigma^d( \omega).
$$
\end{corollary}

\subsubsection*{The Klein--Gordon propagator}

The solution to the Klein--Gordon equation $-\partial^2_{t}u + \Delta u = m^2 u$ in $\R^{d} \times \R$, with initial data $u(x,0)=f_0(x)$, $\partial_t u(x,0)=f_1(x)$ is given by
$$
u(x,t)=e^{i t \sqrt{m^2-\Delta}} f_-(x) + e^{-i t \sqrt{-\Delta}} f_+(x)
$$
where $f_+=\frac{1}{2}(f_0 + i (\sqrt{m^2-\Delta})^{-1} f_1)$ and $f_-=\frac{1}{2}(f_0 - i (\sqrt{m^2-\Delta})^{-1} f_1)$ and 
$$
e^{ \pm i t \sqrt{m^2 - \Delta}} f(x) := \frac{1}{(2\pi)^d} \int_{\R^d} e^{i x \cdot \xi} e^{\pm i t \sqrt{m^2+ |\xi|^2}} \widehat{f} (\xi) \; \ud \xi.
$$
Note that $e^{ \pm i t \sqrt{m^2 - \Delta}} f(x)= (2\pi)^{-d} (\widehat{g}\ud \sigma_{\mathbb{H}^d_m})\: \widehat{\:\:}\:(x,t)$; where $\widehat{g}$ is the lift of $\widehat{f} \sqrt{m^2 + |\cdot|^2}$ to $\mathbb{H}^d_m$. Thus, Theorem \ref{thm:hyp 1} and Corollaries \ref{cor:hyp 1} and \ref{cor:hyp 2} may be re-interpreted in terms of $e^{i t \sqrt{m^2-\Delta}}$; in particular, setting $\mathbf{KG(d)}= (2\pi)^{-4d} \mathbf{C}_{\mathbb{H}^d}$, the estimate in Corollary \ref{cor:hyp 2} reads as
$$
\| (-\Delta_x)^{\frac{2-d}{4}} ( e^{i t \sqrt{m^2-\Delta}} f_1\overline{e^{i t \sqrt{m^2-\Delta}} f_2})  \|_{L^2_{x,t}(\R^{d+1})}^2 \leq \!\! \mathbf{KG(d)} \!\!
\int_{(\R^d)^2} \!\!\!\!\!\!  K_{ \mathbb{H}^d_m} (\xi, \zeta) |\widehat{f}_1(\xi)|^2 |\widehat{f}_2(\zeta)|^2 \phi_m(\xi) \phi_m(\zeta)  \ud \xi  \ud \zeta.
$$

\subsection*{Structure of the paper} Section 2 contains some notation and standard observations which will be useful throughout the paper. In Section 3 we revisit the convolution of weighted measures of circles and hyperbolas. Section 4 contains the proofs of Theorems \ref{thm:sphere1} and \ref{thm:hyp 1} whilst Section 5 is concerned with the derivation of the several corollaries. Finally, we provide a Fourier analytic proof of Theorem \ref{PV thm} in Section 6, together with a further discussion on Fourier bilinear identities associated to paraboloids.

\subsection*{Acknowledgements} The authors would like to thank Pedro Caro for several discussions at early stages of this project and Jon Bennett, Giuseppe Negro and Mateus Sousa for stimulating conversations. They also would like to thank the anonymous referee for a careful reading and many valuable suggestions.


\section{Notation and preliminaries}\label{sec:preliminaries}


\subsection{Fourier transform}\label{subsec:FT}
We work with the normalisation of the Fourier transform
$$
\mathcal{F}(f)(\xi)=\widehat{f}(\xi)=\int_{\R^{n}} e^{i z \cdot \xi} f(z) \, \ud z \qquad \textrm{and} \qquad \mathcal{F}^{-1}(f)(z)=\frac{1}{(2\pi)^n}\int_{\R^{n}} e^{-i z \cdot \xi} f(\xi) \, \ud \xi.
$$
With this normalisation,
$$
\widehat{f \ast g} = \widehat{f} \cdot \widehat{g}, \qquad \widehat{fg}(\xi)=(2\pi)^{-n} \widehat{f} \ast \widehat{g} (\xi), \qquad \widehat{\widehat{f}}(z)=(2\pi)^n\tilde{f}(z), \qquad \textrm{}  \widehat{\bar{\widehat{f}}}(z)=(2\pi)^n\bar{f}(z),
$$
where $\tilde{f}(z):=f(-z)$, Plancherel's theorem adopts the form
$$
\| \widehat{f} \|_{L^2(\R^{n})} = (2 \pi)^{n/2} \| f\|_{L^2(\R^{n})}.
$$
The $n$-dimensional Dirac delta, denoted by $\delta_n$ is understood as
$$
 \delta_n(a) = \frac{1}{(2\pi)^n}\int_{\R^n} e^{i a \cdot z} \, \ud z .
$$


\subsection{$k$-plane transform}\label{subsec:Radon}
The Grassmannian manifold $\CG_{k,n}$ of all $k$-dimensional subspaces of $\R^n$ is equipped with an invariant measure $\ud \mu_{\CG}$ under the action of the orthogonal group. This measure is unique up to a constant, and is chosen to be normalised as
$$
|\mathcal{G}_{k,n}|=\int_{\mathcal{G}_{k,n}} \, \ud \mu_\mathcal{G}(\pi) = \frac{|\bS^{n-1}| \cdots |\bS^{n-k}| }{|\bS^{k-1}| \cdots |\bS^0|}.
$$
Given $\pi \in \mathcal{G}_{k,n}$ and $\xi \in \pi^\perp$, the relation \eqref{eq:FT and kplane} between the $k$-plane transform $T_{k,n}$ and the Fourier transform easily follows from the definition
\begin{equation}\label{FT k-plane}
\CF_y T_{k,n} f (\pi, \xi)= \int_{\pi^\perp} e^{i y \cdot \xi} T_{k,n} f (\pi, y) \, \ud \lambda_{\pi^\perp}(y) = \int_{\pi^\perp} e^{i y \cdot \xi} \int_{\pi} f (x+ y) \, \ud \lambda_\pi(x) \, \ud \lambda_{\pi^\perp}( y )= \widehat{f}(\xi)
\end{equation}
after changing variables $z=x+y$ and noting that $\xi \cdot x =0$ for $\xi \in \pi^\perp$. This and the known identity (see, for instance, \cite[Chapter 2]{Helgason})
\begin{align}
\int_{\bS^{n-1}} f(\omega) \, \ud\sigma^n(\omega)  =  \frac{1}{|\CG_{n-k-1,n-1}|} \int_{\CG_{k,n}}  \int_{\bS^{n-1} \cap \pi^\perp} f(\omega) \, \ud \sigma^n_{\pi^\perp}(\omega) \, \ud \mu_{\mathcal{G}}(\pi), \label{g=1}
\end{align}
yield via Plancherel's theorem and a change to polar coordinates the Plancherel-type identity \eqref{eq:plancherel k-plane} for the $k$-plane transform.


\subsection{Lorentz transformations}\label{subsec:Lorentz}

The Lorentz group $\mathcal{L}$ is defined as the group of invertible linear transformations in $\R^{d+1}$ preserving the bilinear form
$$
(z, u ) \mapsto z_{d+1}u_{d+1} - z_d u_d - \cdots - z_1u_1.
$$
It is well-known that the measure $\ud \sigma_{\bH^d_m}$ is invariant under the action of the subgroup of $\mathcal{L}$ that preserves the hyperboloid $\bH^d_m$, denoted by $\mathcal{L}^+$. More precisely,
$$
\int_{\bH^d_m} f \circ L \, \ud \sigma_{\bH^d_m} = \int_{\bH^d_m} f  \, \ud \sigma_{\bH^d_m} 
$$
for all $L \in \mathcal{L}^+$. It is also a well-known fact that given $P=(\xi, \tau) \in \R^{d+1}$ with $\tau > |\xi|$, there exists a Lorentz transformation $L \in \mathcal{L}^+$ such that $L(\xi, \tau)=(0, \sqrt{\tau^2-|\xi|^2})$; see, for instance, \cite{Quilodran}. For $d=1$, this transformation is given by
\begin{equation}\label{eq:Lorentz}
L \equiv L_{\gamma_P}:=\begin{pmatrix}
\cosh \gamma_P & - \sinh \gamma_P \\
-\sinh \gamma_P & \cosh \gamma_P
\end{pmatrix}, \qquad \textrm{where} \quad \gamma_P:=\ln \sqrt{\frac{\tau+\xi}{\tau-\xi}};
\end{equation}
recall that $P$ may be expressed in hyperbolic coordinates as $P=(\xi,\tau)=(r_P \sinh \gamma_P, r_P \cosh \gamma_P)$, where $r_P:= \sqrt{\tau^2-\xi^2}$. The inverse Lorentz transformation that maps $(0, r_P)$ back to $(\xi, \tau)$ is given by $L_{-\gamma_P}$.


\section{Convolution of weighted measures}

As is discussed in the introduction, a key ingredient in the proofs of Theorems \ref{thm:sphere1} and \ref{thm:hyp 1} is to understand convolutions of two weighted measures associated to concentric circles of different radii in $\R^2$ and to hyperbolas in $\R^2$ with the same perpendicular asymptotes and foci lying on the same line but with different major axis. The computation of such convolutions is standard; see, for instance, \cite{Foschi, CO} for the circular case or \cite{Quilodran, COShyp} for the hyperbolic case. The main feature here is that the convolution is carried with respect to weighted measures, and, since the analysis is restricted to $\R^2$, one can give a precise evaluation of such weights at certain points.


\subsection{Circles}

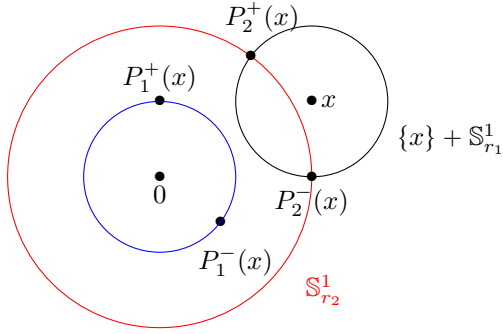
\begin{figure}
\begin{center}

\begin{tikzpicture}

\draw[name path= circle2, red] (0,0) circle (2cm);
\node[right] at (1.8,-1.5) {\textcolor{red}{$ \bS^1_{r_2}$}};

\coordinate (z) at (2,1);

\node at (z) {$\bullet$};
\node[right] at (z) {$x$};

\draw[name path= circle1] (z) circle (1cm);
\node[right] at (3,0.5) {$\{x\} +  \bS^1_{r_1}$};

\path [name intersections={of=circle1 and circle2, by={x2,y2}}];

\coordinate (O) at (0,0);
\node[below] at (O) {$0$};
\node at (O) {$\bullet$};

\node[above] at (1.35,1.8) {$P_2^+(x)$};
\node at (x2) {$\bullet$};

\node[below] at (y2) {$P_2^-(x)$};
\node at (y2) {$\bullet$};

\coordinate (x1) at ( $(z)-(x2)$ );
\coordinate (y1) at ( $(z) - (y2)$);

\node[below] at (1,-0.8) {$P_1^-(x)$};
\node at (x1) {$\bullet$};

\node[above] at (y1) {$P_1^+(x)$};
\node at (y1) {$\bullet$};

\draw[blue] (0,0) circle (1cm);

		\end{tikzpicture}
\caption{The points $P_2^+(x), P_2^-(x) \in \bS^1_{r_2} \cap (\{x\} + \bS^1_{r_1})$ and the points $P_1^+(x):=x-P_2^-(x), P_1^-(x):=x-P_2^+(x) \in  \bS^1_{r_1}$. Note that $P_1^+(x), P_1^-(x) \in \mathbb{S}^1_{r_1} \cap (\{x\} + \mathbb{S}_{r_2}^1)$.}
\label{figure:circles new points 2}
\end{center}
\end{figure}

Given $r \in \R_+$, let $\ud \sigma^2_r$ denote the normalised Lebesgue measure of $\bS^1_{r} \equiv r \bS^1$, that is
$$
\int_{\bS^{1}_{r}} g(\omega) \, \ud \sigma_r^2 (\omega) = \int_{\bS^{1}} g(r \omega) \, \ud \sigma^2 (\omega),
$$
and recall that $\ud \sigma^2(\omega) = \delta_1 (1 - |\omega|) \, \ud \omega = 2\delta_1 (1-|\omega|^2) \, \ud\omega$, where $\ud \omega$ denotes the Lebesgue measure on $\R^2$ and $\delta_n$ denotes the $n$-dimensional Dirac delta.

Given $0 < r_1 \leq r_2$,  the domain of integration in $\ud \sigma_{r_1}^2 \ast \ud \sigma^2_{r_2}(x)$ is $\bS_{r_2}^1 \cap (\{x\} + \bS_{r_1}^1)$. This set is non-empty if and only if $|x|\in [r_2-r_1, r_2+r_1]$ and consists of one point in the \textit{tangent} case $|x|=r_2-r_1$ or $|x|=r_2+r_1$ and of two points otherwise. In the non-empty case, let $v_x \in \bS^1$ denote the $\pi/2$ degrees rotation of $x/|x|$ in the anti-clockwise direction, and let $P_2^+(x)$ and $P_2^-(x)$ denote the points in $\bS_{r_2}^1 \cap (\{x\} + \bS^1_{r_1})$ such that $P_2^+(x) \cdot v_x \geq 0$ and $P_2^-(x) \cdot v_x \leq 0$ respectively; note that $P_2^+(x)=P_2^-(x)$ in the tangent case. Define $P_1^-(x):=x-P_2^+(x) \in \bS^1_{r_1}$ and $P_1^+(x):=x-P_2^-(x) \in \bS^1_{r_1}$; note that $P_1^+(x), P_1^-(x) \in \mathbb{S}^1_{r_1} \cap (\{x\} + \mathbb{S}_{r_2}^1)$. Observe that $P_j^+(x)$ and $P_j^-(x)$ are reflected points one another with respect to the line passing through the origin containing $x$:  see Figure \ref{figure:circles new points 2}.

\begin{lemma}\label{convolution:circles}
Let $r_1, r_2 \in \R$ such that $0 < r_1 \leq r_2$. 
Then
$$
g_1 \ud \sigma_{r_1}^{2} \ast g_2 \ud \sigma^2_{r_2} (x) = \frac{2 g_1(P_1^+(x)) g_2(P_2^-(x)) + 2 g_1(P_1^-(x)) g_2(P_2^+(x)) }{\sqrt{ -(|x|^2 - (r_2+r_1)^2) (|x|^2-(r_2-r_1)^2)}}
$$
if $|x| \in [r_2-r_1, r_2+r_1]$.
\end{lemma}

\begin{proof}
A standard computation shows (see for instance \cite[Lemma 2.2]{Foschi} or \cite[Lemma 5]{CO} for similar arguments)
\begin{align*}
g_1\,  \ud \sigma_{r_1}^2 \ast g_2 \, \ud\sigma_{r_2}^2 (x) & =  \int_{\bS^1} \int_{\bS^1} g_1(r_1 \omega_1) g_2(r_2\omega_2) \delta_2(x-r_1\omega_1 - r_2\omega_2) \, \ud\sigma^2(\omega_1)\, \ud\sigma^2(\omega_2) \\
& = \frac{2}{r_2^2} \int_{\bS^1} \int_{\R^2} g_1(r_1 \omega_1) g_2(r_2 \omega_2) \delta_2 \Big(\frac{x}{r_2}-\frac{r_1}{r_2} \omega_1 - \omega_2 \Big) \delta_1(1-|\omega_2|^2) \, \ud\sigma^2(\omega_1) \, \ud \omega_2 \\
& = \frac{1}{r_1 |x|} \int_{\bS^1} g_1(r_1 \omega_1) g_2(x-r_1\omega_1)  \delta_1\Big(\frac{r_2^2}{2r_1|x|}- \frac{|x|}{2r_1} - \frac{r_1}{2|x|} + \frac{ x}{|x|} \cdot \omega_1\Big) \, \ud\sigma^2(\omega_1) \\
& =:\textrm{I}^+(x) + \textrm{I}^-(x),
\end{align*}
where $\textrm{I}^+(x)$ corresponds to the integration over $\bS^1_+(x):=\{ \omega \in \bS^1 : x \cdot \omega \geq0\}$ and $\textrm{I}^-(x)$ to the integration over $\bS^1_-(x):= \bS^1 \backslash \bS^1_+(x) = \{ \omega \in \bS^1 : x \cdot \omega \leq 0\}$.

Denoting by $\alpha_x$ the clockwise angle between $e_1$ and $x$ and $P_x(u)=(\cos(\alpha_x+\arccos(u)), \sin(\alpha_x + \arccos(u)) )$, the expression for $\textrm{I}^+(x)$ becomes, after a change of variable,
\begin{align*}
\textrm{I}^+(x) &= \frac{1}{r_1 |x|} \int_{-1}^1 \delta_1 \Big(\frac{r_2^2}{2r_1|x|}- \frac{|x|}{2r_1} - \frac{r_1}{2|x|} + u\Big) (1-u^2)^{-1/2} g_1(r_1 P_x(u) ) g_2(x- r_1P_x(u) ) \, \ud u \\
& = \frac{1}{r_1 |x|}\Big(1-\Big(\frac{|x|^2+r_1^2 - r_2^2}{2r_1|x|} \Big)^2 \Big)^{-1/2}g_1(P_1^+(x)) g_2(P_2^-(x))\chi_{\{r_2-r_1 \leq |x| \leq r_2+r_1\}} (x)\\
& = \frac{2g_1(P_1^+(x)) g_2(P_2^-(x))}{\sqrt{-(|x|^2 - (r_2+r_1)^2) (|x|^2-(r_2-r_1)^2)}} \chi_{\{r_2-r_1 \leq |x| \leq r_2+r_1\}} (x)
\end{align*}
noting that $r_1P_x(u)=P_1^+(x)$ after integrating in $u$. Indeed, if $y \in \mathbb{S}_{r_1}^1 \cap (\{x\} + \mathbb{S}^1_{r_2} )$ one has $|y-x|^2=r_2^2$ and $|y|^2=r_1^2$. Therefore $\frac{y}{|y|} \cdot \frac{x}{|x|}= \frac{|x|^2+r_1^2-r_2^2}{2 r_1 |x|} = u$, and $r_1 P_x(u)=P_1^+(x)$ follows from noting that $P_1^+(x) \in \mathbb{S}_{r_1}^1 \cap (\{x\} + \mathbb{S}^1_{r_2} )$ and $P_1^+(x) \cdot v_x \geq 0$. Arguing similarly for $\textrm{I}^-(x)$ concludes the proof.
\end{proof}


\subsection{Hyperbolas}

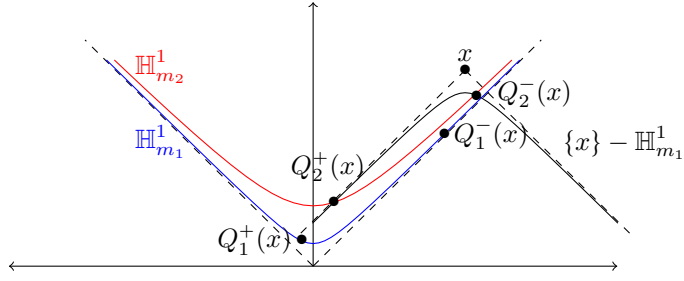
\begin{figure}

\begin{tikzpicture}

    \draw [<->]  (-4,0) to (4,0) ;

\draw [->] (0,0) to (0,3.5);        
        
        \draw [-, dashed] (0,0) to (3,3);
        \draw [-, dashed] (0,0) to (-3,3);

    \pgfmathsetmacro{\a}{0.8}
    \draw[name path=hyp1 , red] plot[domain=-1.9:1.9] ({\a*sinh(\x)},{\a*cosh(\x)});
    
\node[red, above] at (-2,2.3) {$\bH_{m_2}^1$};     
    
    \pgfmathsetmacro{\as}{0.3}
    \draw[name path=hyp2] plot[domain=-2.6:2.6] ({2+\as*sinh(\x)},{2.6-\as*cosh(\x)});
    
\node[above] at (4.1,1.3) {$\{x\} - \bH_{m_1}^1$};

    \draw[blue] plot[domain=-2.9:2.9] ({\as*sinh(\x)},{\as*cosh(\x)});

\node[blue, below] at (-2,2) {$\bH_{m_1}^1$};     

\node[above] at (2,2.6) {$x$};
\node at (2,2.6) {$\bullet$};

\draw [-, dashed] (2,2.6) to (-0.2,0.4);
\draw [-, dashed] (2,2.6) to (4.2,0.4);

\coordinate (x) at (2,2.6);

\path [name intersections={of=hyp1 and hyp2, by={x1,z1}}];

\node at (x1) {$\bullet$};
\node[above] at (0.2,1) {$Q_2^+(x)$};


\node at (2.15,2.25) {$\bullet$};
\node[right] at (2.3,2.3) {$Q_2^-(x)$};

\coordinate (x2) at ( $(x)-(x1)$ );
\coordinate (y2) at ( $(x) - (2.15,2.25)$);

\node at (x2) {$\bullet$};
\node[right] at (x2) {$Q_1^-(x)$};

\node at (y2) {$\bullet$};
\node[left] at (y2) {$Q_1^+(x)$};

\end{tikzpicture}
\caption{The points $Q_2^+(x), Q_2^-(x) \in \bH^1_{m_2} \cap (\{x\} - \bH^1_{m_1})$ and the points $Q_1^+(x):=x-Q_2^-(x), Q_1^-(x):=x-Q_1^+(x) \in \bH^1_{m_1}$.}
\label{figure:hyp new points}
\end{figure}

Consider the Lorentz invariant measure $\ud \sigma_{\mathbb{H}^1_{m}}$ defined in \eqref{eq:Lorentz inv measure}. Given $0 < m_1 \leq m_2$, the domain of integration in $\ud \sigma_{\mathbb{H}^1_{m_1}} \ast \:\: \ud \sigma_{\mathbb{H}^1_{m_2}} (x)$ is $\mathbb{H}^1_{m_2} \cap (\{x\} - \mathbb{H}^1_{m_1})$. Reasoning as in the previous case, this set is non-empty if and only if $\sqrt{x_2^2-x_1^2} \geq m_1+m_2$ and consists of one single point in the \textit{tangent} case $\sqrt{x_2^2-x_1^2} = m_1+m_2$ and of two points otherwise; here $x=(x_1,x_2) \in \R^2$. In the non-empty case, let $Q_2^+(x)$ and $Q_2^-(x)$ denote the points in $\mathbb{H}^1_{m_2} \cap (\{x\} - \mathbb{H}^1_{m_1})$ such that $\big(Q_2^+(x) - x\big)\cdot e_1 \geq 0$ and $\big(Q_2^-(x) - x\big)\cdot e_1 \leq 0$ respectively; of course, $Q_2^+(x)=Q_2^-(x)$ in the tangent case. Define $Q_1^+(x)=x-Q_2^-(x) \in \mathbb{H}^1_{m_1}$ and $Q_1^-(x)=x-Q_2^+(x) \in \mathbb{H}^1_{m_1}$ (see Figure \ref{figure:hyp new points}).

\begin{lemma}\label{convolution:hyp}
Let $m_1,m_2 \in \R$ such that $0< m_1 \leq m_2$. For each $x=(x_1,x_2) \in \R^2$ such that $x_2^2 \geq x_1^2$ one has
$$
g_1 \ud \sigma_{\mathbb{H}^1_{m_1}} \! \ast g_2 \ud  \sigma_{\mathbb{H}^1_{m_2}} (x) = \frac{2 g_1(Q_1^+(x)) g_2(Q_2^-(x)) + 2g_1(Q_1^-(x)) g_2(Q_2^+(x)) }{\sqrt{(x_2^2-x_1^2)^2-2(x_2^2-x_1^2)(m_1^2+m_2^2) + (m_1^2-m_2^2)^2}} \chi_{\{ \sqrt{x_2^2-x_1^2} \geq m_1+m_2 \}}(x).
$$
\end{lemma}

\begin{proof}
By invariance of the measure $\ud \sigma_{\mathbb{H}^1_m}$ under Lorentz transformations, it suffices to prove the above identity for $x=(0,z)$. Indeed, note that if $L_x \in \mathcal{L}^+$ is the Lorentz transformation satisfying $L_x (x)=(0, z)=(0,\sqrt{x_2^2-x_1^2})$, then
\begin{align*}
g_1 \ud \sigma_{\mathbb{H}^1_{m_1}}  \ast g_2 \ud \sigma_{\mathbb{H}^1_{m_2}} (x)  & = \int_{\mathbb{H}_{m_1}^1} \! \int_{\mathbb{H}_{m_2}^1} g_1(\omega) g_2(\nu)  \delta_2( x - \omega - \nu ) \, \ud \sigma_{\mathbb{H}^1_{m_1}}(\omega) \, \ud \sigma_{\mathbb{H}^1_{m_2}} (\nu)\\
& = \int_{\mathbb{H}_{m_1}^1} \! \int_{\mathbb{H}_{m_2}^1} g_1(L_x^{-1}(\omega)) g_2(L_x^{-1}(\nu)) \delta_2( (0,z) - \omega - \nu )  \, \ud \sigma_{\mathbb{H}^1_{m_1}}(\omega) \, \ud \sigma_{\mathbb{H}^1_{m_2}} (\nu) \\
& = h_1 \ud \sigma_{\mathbb{H}^1_{m_1}}  \ast h_2 \ud \sigma_{\mathbb{H}^1_{m_2}} (0,z)
\end{align*}
where $h_j = g_j \circ L_x^{-1}$; the reduction to the vertical axis then follows from noting that $h_j (Q_\ell^{\pm} (0,z))= g_j (Q_\ell^{\pm}(x))$ for $j, \ell=1,2$. Next,
\begin{align*}
h_1 \ud \sigma_{\mathbb{H}^1_{m_1}}  \! \ast h_2 \ud \sigma_{\mathbb{H}^1_{m_2}} (0,z)  & = \int_{\mathbb{H}_{m_1}^1} \! \int_{\mathbb{H}_{m_2}^1} \! h_1(\omega) h_2(\nu) \delta_2( (0,z) - (\omega_1,\omega_2) - (\nu_1,\nu_2) ) \, \ud \sigma_{\mathbb{H}^1_{m_1}}(\omega) \, \ud \sigma_{\mathbb{H}^1_{m_2}} (\nu)\\
& =  \int_{\R} h_1(\omega_1,\phi_{m_1}(\omega_1)) h_2(-\omega_1, \phi_{m_2}( \omega_1))  \frac{\delta_1(z-\phi_{m_1}(\omega_1) - \phi_{m_2}(\omega_1))}{\phi_{m_1}(\omega_1) \phi_{m_2}(\omega_1)} \,  \ud \omega_1.
\end{align*}
Splitting $\R=\R_- \cup \R_+$ and doing the change of variables
$$
v= \phi_{m_1}(\omega_1) + \phi_{m_2}(\omega_1), \qquad  \textrm{with} \quad \frac{\ud\omega_1}{\phi_{m_1}(\omega_1) \phi_{m_2}(\omega_1)} = \frac{\ud v}{\omega_1 v},
$$
on each half-line one has
\begin{align*}
h_1 \ud \sigma_{\mathbb{H}^1_{m_1}} \ast h_2 \ud \sigma_{\mathbb{H}^1_{m_2}} (0,z) = \int_{m_1 +m_2}^\infty   \Big( h_1( & \omega_1,\phi_{m_1}  (\omega_1)) h_2(-\omega_1, \phi_{m_2}(\omega_1)) \\
&+ h_1(-\omega_1,\phi_{m_1}(\omega_1)) h_2(\omega_1, \phi_{m_2}(\omega_1)) \Big)   \delta_1(z-v) \frac{\ud v}{\omega_1 v}
\end{align*}
where $\omega_1$ above is the function of $v$
\begin{equation*}
\omega_1=\omega_1(v):= \frac{\sqrt{v^4 -2v^2(m_1^2+m_2^2) + (m_1^2-m_2^2)^2}}{2v}.
\end{equation*}
Noting that $Q_1^{\pm}(0,v)=( \pm \omega_1(v), \phi_{m_1} (\omega_1(v)) )$ and  $Q_2^{\pm}(0,v)=( \mp \omega_1(v), \phi_{m_2} (\omega_1(v)))$, one has
\begin{align*}
h_1 \ud \sigma_{\mathbb{H}^1_{m_1}} \ast h_2 \ud \sigma_{\mathbb{H}^1_{m_2}} (0,z) & = \frac{2 h_1(Q^+_1(0,z)) h_2(Q^-_{2}(0,z)) + 2 h_1(Q^-_1(0,z)) h_2(Q^+_2(0,z)) }{\sqrt{z^4-2z^2(m_1^2+m_2^2) + (m_1^2-m_2^2)^2}} 1_{\{z \geq m_1+m_2\}},
\end{align*}
completing the proof.
\end{proof}


\section{The proofs of Theorems \ref{thm:sphere1} and \ref{thm:hyp 1}}


\subsection{Proof of Theorem \ref{thm:sphere1}}

For simplicity we work on the unit sphere $r=1$; the result for $\bS^{n-1}_r$ follows analogously. Given $\pi\in \CG_{n-2, n}$, let $\pi^\perp$ denote its orthogonal subspace. For each $\xi \in \R^n$, write $\xi=\xi^\pi + \xi^\perp$, where $\xi^\pi \in \pi$ and $\xi^\perp \in \pi^\perp$, and let $r_{\xi}^\pi:=\sqrt{1-|\xi^\pi|^2}$. Given $x \in \pi$ and $y \in \pi^\perp$,
\begin{align*}
\widehat{g_j \ud \sigma^n} (x+y) &= \int_{\bS^{n-1}} e^{i (x+y) \cdot \xi} g_j(\xi) \, \ud \sigma^n (\xi) \\
& = \int_{|\xi^\pi| \leq 1} e^{i x \cdot \xi^\pi} \int_{r_{\xi}^\pi \bS^1} e^{i y \cdot \xi^\perp} g_j(\xi^\pi+\xi^\perp) \, \ud \sigma^\perp_{r_{\xi}^{\pi}} (\xi^\perp) \, \ud \lambda_{\pi} (\xi^\pi) \\
&=  \int_{|\xi^\pi| \leq 1} e^{i x \cdot \xi^\pi} \mathcal{F}^\perp ( g_{j,\xi^\pi} \ud \sigma_{r_{\xi}^{\pi}}^\perp  ) (y) \, \ud \lambda_\pi (\xi^\pi)
\end{align*}
where $g_{j,\xi^\pi}(\omega):= g_j (\xi^\pi + \omega)$, $\mathcal{F}^{\perp}$ denotes the Fourier transform in $\pi^\perp$ and $\ud \sigma_{r_{\xi}^{\pi}}^\perp$ denotes the induced normalised Lebesgue measure of $r_{\xi}^{\pi}\bS^1$ in $\pi^\perp$, which can be, of course, identified with $\ud \sigma_{r_{\xi}^{\pi}}^2$. Then, by Plancherel's theorem in $\pi$ (with the  normalisation for the Fourier transform taken  in $\S$\ref{sec:preliminaries}),
\begin{align*}
T_{n-2,n} (\widehat{g_1 \ud \sigma^n} \overline{\widehat{g_2 \ud \sigma^n}}) (\pi, y) &= \int_{\pi} \widehat{g_1 \ud \sigma^n}(x+y) \overline{\widehat{g_2 \ud \sigma^n}}(x+y) \, \ud \lambda_{\pi} (x) \\
& = (2\pi)^{n-2} \int_{|\xi^\pi| \leq 1} \mathcal{F}^\perp ( g_{1,\xi^\pi} \ud \sigma_{r_{\xi}^{\pi}}^\perp  ) (y) \overline{ \mathcal{F}^\perp ( g_{2,\xi^\pi} \ud \sigma_{r_{\xi}^{\pi}}^\perp  )} (y)  \, \ud \lambda_\pi (\xi^\pi).
\end{align*}
A further application of Plancherel's theorem in $\pi^\perp$ yields
\begin{align*}
\int_{\pi^\perp}   |(-\Delta_y)^{1/4} & T_{n-2,n} (\widehat{g_1 \ud \sigma^n} \overline{\widehat{g_2 \ud \sigma^n}}) (\pi, y)|^2 \, \ud \lambda_{\pi^\perp} (y)  \\
&= (2\pi)^{2(n-1)}  \int_{|\xi^\pi| \leq 1}  \int_{|\zeta^\pi| \leq 1} A_{\mathbb{S}^{n-1}, \pi} (\xi^\pi, \zeta^\pi ) \, \ud \lambda_\pi (\xi^\pi) \, \ud \lambda_\pi (\zeta^\pi).
\end{align*}
where 
$$
A_{\mathbb{S}^{n-1}, \pi} (\xi^\pi, \zeta^\pi ):= \int_{\pi^\perp} |v| \big( \tilde{g}_{1,\xi^\pi} \ud \sigma_{r_{\xi}^{\pi}}^\perp  \ast^\perp \bar g_{2,\xi^\pi} \ud \sigma_{r_{\xi}^{\pi}}^\perp \big)   (v) \big( \overline{\tilde{g}_{1,\zeta^\pi} \ud \sigma_{r_{\zeta}^{\pi}}^\perp  \ast^\perp \bar g_{2,\zeta^\pi} \ud \sigma_{r_{\zeta}^{\pi}}^\perp } \big)   (v)   \, \ud \lambda_{\pi^\perp} (v)
$$
and $\tilde{g}(\cdot) = g(-\,  \cdot)$. For fixed $\xi^\pi$, $\zeta^\pi$ with $|\xi^\pi| \leq 1$ and $|\zeta^\pi| \leq 1$, $A_{\mathbb{S}^{n-1}, \pi}(\xi^\pi, \zeta^\pi)$ equals to
\begin{equation}\label{eq:innermost integral Sigma}
\int_{(r_\xi^\pi \bS^1)^2 \times (r_\zeta^\pi \bS^1)^2} \!\!\!\!\!\!\!\!\!\!\!\!\!\!\!\!\!\!\!\!  |\xi^\perp -\eta^\perp|^{1/2} |\zeta^\perp-\mu^\perp|^{1/2} g_{1, \xi^\pi} (\xi^\perp) \bar g_{2, \xi^\pi} (\eta^\perp)  g_{2, \zeta^\pi}(\zeta^\perp) \bar g_{1, \zeta^\pi}(\mu^\perp) \, \ud \Sigma^{\perp}_{\xi^\pi, \zeta^\pi}(\xi^\perp, \eta^\perp,\zeta^\perp,\mu^\perp)
\end{equation}
where
$$
\ud \Sigma^{\perp}_{\xi^\pi, \zeta^\pi}(\xi^\perp, \eta^\perp,\zeta^\perp,\mu^\perp) := \delta (\xi^\perp + \zeta^\perp -\eta^\perp-\mu^\perp) \, \ud \sigma_{r_{\xi}^{\pi}}^\perp(\xi^\perp) \, \ud \sigma_{r_{\xi}^{\pi}}^\perp(\eta^\perp) \, \ud \sigma_{r_{\zeta}^{\pi}}^\perp(\zeta^\perp) \, \ud \sigma_{r_{\zeta}^{\pi}}^\perp(\mu^\perp).
$$
Observe that one may rewrite the above integral as
$$
A_{\mathbb{S}^{n-1}, \pi}(\xi^\pi, \zeta^\pi)= \int_{r_\xi^\pi \bS^1 \times r_\zeta^\pi \bS^1} \!\!\!  g_{1,\xi^\pi} (\xi^\perp) g_{2, \zeta^\pi}(\zeta^\perp) \big( h_{2, \xi} \ud \sigma^\perp_{r_\xi^\pi} \ast^\perp h_{1, \zeta} \ud \sigma^{\perp}_{r_\zeta^\pi} \big) (\xi^\perp + \zeta^\perp) \, \ud \sigma_{r_{\xi}^{\pi}}^\perp(\xi^\perp)   \, \ud \sigma_{r_{\zeta}^{\pi}}^\perp(\zeta^\perp) 
$$
where $h_{2, \xi}(\eta^\perp):= \bar g_{2, \xi^\pi}(\eta^\perp) |\xi^\perp - \eta^\perp|^{1/2}$ and similarly for $h_{1, \zeta}$. As $\pi^\perp \cong \R^2$,  assuming without loss of generality that $r_{\xi}^\pi \leq r_{\zeta}^\pi$, one can appeal to Lemma \ref{convolution:circles} to evaluate
\begin{equation}\label{conv:evaluation circles}
 \big( h_{2, \xi} \ud \sigma^\perp_{r_\xi^\pi} \ast^\perp h_{1,\zeta} \ud \sigma^{\perp}_{r_\zeta^\pi} \big) (\xi^\perp + \zeta^\perp)  =  \frac{2 h_{2,\xi}(\xi^\perp) h_{1,\zeta}(\zeta^\perp) + 2 h_{2,\xi}(\tilde{\xi}^\perp)h_{1,\zeta}(\tilde{\zeta}^\perp) }{\sqrt{ -(|\xi^\perp + \zeta^\perp|^2 - (r_\zeta^\pi+r_\xi^\pi)^2) (|\xi^\perp + \zeta^\perp|^2-(r_{\zeta}^\pi-r_\xi^\pi)^2)}}
\end{equation}
after noting that if $x=\xi^\perp + \zeta^\perp$ then $(P_1^{+}(x), P_1^-(x), P_2^+(x), P_2^-(x))=(\xi^\perp, \tilde{\xi}^\perp, \zeta^\perp, \tilde{\zeta}^\perp)$, where $\tilde{\xi}^\perp, \tilde{\zeta}^\perp \in \pi^\perp$ are the reflected points of $\xi^\perp$ and $\zeta^\perp$ with respect to $\xi^\perp + \zeta^\perp$. Note that the implicit support condition $r_{\zeta}^\pi - r_\xi^\pi \leq |\xi^\perp+\zeta^\perp| \leq r_\zeta^\pi + r_\xi^\pi$ in \eqref{conv:evaluation circles} always holds under the assumption $r_{\xi}^\pi \leq r_{\zeta}^\pi$. Observe that $h_{2,\xi}(\xi^\perp)=h_{1,\zeta}(\zeta^\perp)=0$, so manipulating the denominator one has
\begin{align}\label{evaluation conv circles proof}
 \big( h_{\xi} \ud \sigma^\perp_{r_\xi^\pi} \ast^\perp h_{\zeta} \ud \sigma^{\perp}_{r_\zeta^\pi} \big) (\xi^\perp + \zeta^\perp)  & = \left( \frac{  |\xi^\perp- \tilde{\xi}^\perp| |\zeta^\perp- \tilde{\zeta}^\perp|}{(r_\xi^\pi r_\zeta^\pi)^2 -(\xi^\perp \cdot \zeta^\perp)^2 } \right)^{1/2}  \bar g_{2,\xi^\pi}(\tilde{\xi}^\perp)\bar g_{1,\zeta^\pi}(\tilde{\zeta}^\perp)
\end{align}
for all $\xi^\perp \in r_\xi^\pi \bS^1$ and $\zeta^\perp \in r_\zeta^\pi  \bS^1$. Next note that $|\xi^\perp \wedge \zeta^\perp|^2 = (r_\xi^\pi r_\zeta^\pi)^2 -(\xi^\perp \cdot \zeta^\perp)^2$, but also $|\xi^\perp \wedge \zeta^\perp|^2 = \frac{1}{4} |\xi^\perp + \zeta^\perp|^2  |\xi^\perp- \tilde{\xi}^\perp| |\zeta^\perp- \tilde{\zeta}^\perp|$, as the points satisfy the relation $\xi^\perp+\zeta^\perp= \tilde{\xi}^\perp + \tilde{\zeta}^\perp$. Thus,
$$
 \left( \frac{  |\xi^\perp- \tilde{\xi}^\perp| |\zeta^\perp- \tilde{\zeta}^\perp|}{(r_\xi^\pi r_\zeta^\pi)^2 -(\xi^\perp \cdot \zeta^\perp)^2 } \right)^{1/2} = \frac{2}{|\xi^\perp+ \zeta^\perp|},
 $$
and combining the above estimates one obtains
\begin{align*}
&\int_{\pi^\perp}  |(-\Delta_y)^{1/4}  T_{n-2,n} (\widehat{g_1 \ud \sigma^n} \overline{\widehat{g_2 \ud \sigma^n}}) (\pi, y)|^2 \, \ud \lambda_{\pi^\perp} (y) \\
&= (2\pi)^{2(n-1)}  \int_{|\xi^\pi| \leq 1}  \int_{|\zeta^\pi| \leq 1}  \int_{r_\xi^\pi \bS^1 \times r_\zeta^\pi \bS^1} K_{\pi, \bS^{n-1}}(\xi,\zeta)g_{1,\xi^\pi} (\xi^\perp) g_{2,\zeta^\pi}(\zeta^\perp)\bar g_{2,\xi^\pi}(\tilde{\xi}^\perp)\bar g_{1,\zeta^\pi}(\tilde{\zeta}^\perp)  \, \ud \Sigma_\pi (\xi, \zeta) \\
& = (2\pi)^{2(n-1)} \int_{(\bS^{n-1})^2} K_{\pi, \bS^{n-1}}(\xi,\zeta)g_1 (\xi) g_2(\zeta)\bar g_2(\xi^\pi+\tilde{\xi}^\perp)\bar g_1(\zeta^\pi+\tilde{\zeta}^\perp)  \, \ud \sigma^n(\xi)   \, \ud \sigma^n(\zeta) ,
\end{align*}
completing the proof of Theorem \ref{thm:sphere1}; above $\ud \Sigma_\pi (\xi, \zeta):= \, \ud \sigma_{r_{\xi}^{\pi}}^\perp(\xi^\perp)   \, \ud \sigma_{r_{\zeta}^{\pi}}^\perp(\zeta^\perp)  \, \ud \lambda_\pi (\xi^\pi) \, \ud \lambda_\pi (\zeta^\pi)$. \qed


\subsection{Proof of Theorem \ref{thm:hyp 1}} Given $\omega \in \bS^{d-1}_+$ and $\pi= \langle \omega \rangle^\perp \in \mathcal{G}_{d-1,d}$ write, for each $\xi \in \R^d$, $\xi= \xi^\pi + \xi^\omega \omega$, where $\xi^\omega = \xi \cdot \omega $ and let $m_{\xi}^\pi:=\sqrt{m^2+|\xi^\pi|^2}$. Given $s \in \R$ and $x \in \pi$,
\begin{align*}
\widehat{g_j \ud \sigma_{\bH_m^{d}}} (x+ s \omega, t) &= \int_{\R^d} e^{i (x+s\omega) \cdot \xi + i t \sqrt{m^2+|\xi|^2}} f_j(\xi) \frac{ \ud \xi}{\sqrt{m^2+|\xi|^2}} \\
& = \int_{\pi} e^{i x \cdot \xi^\pi} \int_{\R}  e^{i s \xi^\omega  + i t \sqrt{m^2+|\xi^\pi|^2 + |\xi^\omega|^2}} f_j ( \xi^\pi + \xi^\omega \omega) \frac{\ud \xi^\omega }{\sqrt{m^2+ |\xi^\pi|^2 + |\xi^\omega|^2}} \, \ud \lambda_{\pi}(\xi^\pi) \\
& = \int_{\pi} e^{i x \cdot \xi^\pi} \mathcal{F}^2 (g_{j,\xi^\pi} \ud \sigma_{\bH^1_{m_\xi^\pi}}) (s,t) \, \ud \lambda_{\pi}(\xi^\pi),
\end{align*}
where $f_{j,\xi^\pi} (\nu) := f_j (\xi^\pi+\nu \omega)$ for all $\nu \in \R$ and $g_{j, \xi^\pi}$ denotes the lift of $f_{j,\xi^\pi}$ to $\bH^1_{m_\xi^\pi}$, and $\mathcal{F}^2$ denotes the $2$-dimensional Fourier transform. Reasoning as in the proof of Theorem \ref{thm:sphere1},
\begin{align*}
\int_{\R}& \int_{\R}  |\partial_s^{1/2} \mathcal{R} \big( \widehat{g_1 \ud \sigma_{\bH_m^d})}(\cdot, t) \overline{\widehat{g_2 \ud \sigma_{\bH_m^d}}}(\cdot, t)\big) (\omega, s) |^2 \, \ud s \, \ud t  =(2\pi)^{2d} \int_{\pi} \int_{\pi} A_{\mathbb{H}^d_m, \pi}(\xi^\pi, \zeta^\pi)  \, \ud \lambda_{\pi}(\xi^\pi) \, \ud \lambda_{\pi}(\zeta^\pi),
\end{align*}
where 
$$
A_{\mathbb{H}^d_m, \pi}(\xi^\pi, \zeta^\pi):= \int_{\R} \int_\R |v|(\tilde{g}_{1,\xi^\pi} \ud \sigma_{\bH^1_{m_\xi^\pi} }\ast^2 \bar g_{2,\xi^\pi} \ud \sigma_{\bH^1_{m_\xi^\pi}})(v,\tau)  \overline{(\tilde{g}_{1,\zeta^\pi} \ud \sigma_{\bH^1_{m_\zeta^\pi} }\ast^2 \bar g_{2,\zeta^\pi} \ud \sigma_{\bH^1_{m_\zeta^\pi}})} (v,\tau) \, \ud v \, \ud \tau.
$$
Note that
$$
A_{\mathbb{H}^d_m, \pi} (\xi^\pi, \zeta^\pi)= \int_{\R} \int_{\R } f_{1,\xi^\pi}(\xi^\omega) f_{2,\zeta^\pi} (\zeta^\omega) \big( H_{2,\xi} \ud \sigma_{\bH^1_{m_\xi^\pi}} \ast H_{1,\zeta} \ud \sigma_{\bH^1_{m_\zeta^\pi}} \big) (P_{\xi,\zeta,\omega} ) \, \frac{\ud \xi^\omega}{\phi_{m_\xi^\pi}(\xi^\omega)} \, \frac{\ud \zeta^\omega}{\phi_{m_\zeta^\pi}(\zeta^\omega)}, 
$$
where $H_{2,\xi}$ is the lift of $h_{2,\xi}(\eta):= \bar f_{2,\xi^\pi}(\eta)|\xi^\omega-\eta|^{1/2}$ to $\bH^1_{m_\xi^\pi}$ (similarly for $H_{1,\zeta}$) and $P_{\xi, \zeta, \omega}$ denotes the point
$$
P_{\xi,\zeta,\omega}:=\big(\xi^\omega+\zeta^\omega, \phi_{m_\xi^\pi}(\xi^\omega) + \phi_{m_\zeta^\pi}(\zeta^\omega)\big).
$$
Denoting by $r_P$ the hyperbolic radius of $P_{\xi,\zeta,\omega}$, that is, $r_P^2=(\phi_{m_\xi^\pi}(\xi^\omega) + \phi_{m_\zeta^\pi}(\zeta^\omega) )^2 - (\xi^\omega+\zeta^\omega)^2$, Lemma \ref{convolution:hyp} yields
\begin{equation}\label{eq:proof hyp 1}
\big( H_{2,\xi} \ud \sigma_{\bH^1_{m_\xi^\pi}} \ast H_{1,\zeta} \ud \sigma_{\bH^1_{m_\zeta^\pi}} \big) (P_{\xi,\zeta,\omega}) = \frac{2 h_{2,\xi}(\xi^\omega) h_{1,\zeta}(\zeta^\omega) + 2h_{2,\xi}(\tilde{\xi}^\omega) h_{1,\zeta}(\tilde{\zeta}^\omega) }{\sqrt{ r_P^4 - 2r_P^2 \big( (m_\xi^\pi)^2 + (m_\zeta^\pi)^2  \big) + \big( (m_\xi^\pi)^2 - (m_\zeta^\pi)^2  \big)^2}}
\end{equation}
where $(\tilde{\xi}^\omega, \phi_{m_\xi^\pi} (\tilde{\xi}^\omega)) = Q_1^-( P_{\xi,\zeta,\omega} ) \in \bH^1_{m_\xi^\pi}$ and $(\tilde{\zeta}^\omega, \phi_{m_\zeta^\pi} (\tilde{\zeta}^\omega)) = Q_2^+( P_{\xi,\zeta,\omega}) \in \bH^1_{m_\zeta^\pi}$. After an algebraic manipulation and noting that $h_{1,\zeta}(\zeta^\omega)=h_{2,\xi}(\xi^\omega)=0$, \eqref{eq:proof hyp 1} becomes
$$
\big( H_{2,\xi} \ud \sigma_{\bH^1_{m_\xi^\pi}} \ast H_{1,\zeta} \ud \sigma_{\bH^1_{m_\zeta^\pi}} \big) (P_{\xi,\zeta,\omega}) = \frac{|\xi^\omega-\tilde{\xi}^\omega|^{1/2} \bar f_{2,\xi^\pi}(\tilde{\xi}^\omega) |\zeta^\omega -\tilde{\zeta}^\omega|^{1/2} \bar f_{1,\zeta^\pi} (\tilde{\zeta}^\omega)} {|\xi^\omega \phi_{m_{\zeta}^\pi}(\zeta^\omega) - \zeta^\omega \phi_{m_\xi^\pi}(\xi^\omega)|}.
$$
Putting all the estimates together as in the proof of Theorem \ref{thm:sphere1} concludes the proof. \qed

\begin{remark}
As the points in the pairs $(Q_1^+(0,z), Q_1^-(0,z))$ and $(Q_2^+(0,z), Q_2^-(0,z))$ are symmetric with respect to the vertical axis, it is a simple exercise to obtain an expression for $\tilde{\xi}^\omega$ and $\tilde{\zeta}^\omega$ via Lorentz transformations. Indeed, let $\gamma_P$ denote the hyperbolic angle of $P_{\xi,\zeta,\omega}$ and let $L_{\gamma_P}$ denote, as in \eqref{eq:Lorentz}, the Lorentz transformation such that $L_{\gamma_P} (P_{\xi,\zeta,\omega})=(0,r_P)$. Then
\begin{align*}
Q_{1}^+(0,r_P)&=L_{\gamma_P} (\xi^\omega, \phi_{m_\xi^\pi}(\xi^\omega))=(m_\xi^\pi \sinh (\gamma_\xi - \gamma_P) , m_\xi^\pi \cosh(\gamma_\xi-\gamma_P)) \\
 Q_{2}^-(0,r_P)&=L_{\gamma_P} (\zeta^\omega, \phi_{m_\zeta^\pi}(\zeta^\omega))=(m_\zeta^\pi \sinh (\gamma_\zeta - \gamma_P), m_\zeta^\pi \cosh(\gamma_\zeta-\gamma_P) ).
\end{align*}
Clearly, 
\begin{align*}
Q_{1}^-(0,r_P)&=(-m_\xi^\pi \sinh (\gamma_\xi - \gamma_P), m_\xi^\pi \cosh(\gamma_\xi-\gamma_P) ) \\
 Q_{2}^+(0,r_P)&=(-m_\zeta^\pi \sinh (\gamma_\zeta - \gamma_P), m_\zeta^\pi \cosh(\gamma_\zeta-\gamma_P) )
\end{align*}
and
\begin{align*}
Q_1^-(P_{\xi,\zeta,\omega}) & = L_{-\gamma_P} (Q_1^-(0,r_P)) = (m_{\xi}^\pi \sinh(2\gamma_P-\gamma_\xi), m_\xi^\pi \cosh(2\gamma_P - \gamma_\xi)) \\
Q_2^+(P_{\xi,\zeta,\omega}) & = L_{-\gamma_P} (Q_2^+(0,r_P)) = (m_{\zeta}^\pi \sinh(2\gamma_P-\gamma_\zeta), m_\zeta^\pi \cosh(2\gamma_P - \gamma_\zeta)),
\end{align*}
so $\tilde{\xi}^\omega = m_{\xi}^\pi \sinh(2\gamma_P-\gamma_\xi)$ and $\tilde{\zeta}^\omega=m_{\zeta}^\pi \sinh(2\gamma_P-\gamma_\zeta)$. In particular, this allows one to rewrite the kernel as
$$
K_{\omega, \bH^d_m}(\xi,\zeta)= \frac{\big( m_{\xi}^\omega m_\zeta^\omega |\sinh(\gamma_\xi-\gamma_P)| |\cosh(\gamma_\xi+\gamma_P)| |\sinh(\gamma_\zeta-\gamma_P)| |\cosh(\gamma_\zeta+\gamma_P)|   \big)^{1/2}}{m_\xi^\omega m_\zeta^\omega |\sinh(\gamma_\xi-\gamma_\zeta)|}.
$$
\end{remark}

\begin{remark}
Note that
$$
|\xi^\omega - \tilde{\xi}^\omega| = |(L_{\gamma_P}^{-1} (L_{\gamma_P}(\xi^\omega, \phi_{m_\xi^\pi}(\xi^\omega)) - L_{\gamma_P}((\tilde{\xi}^\omega, \phi_{m_\xi^\pi}(\tilde{\xi}^\omega)) )  )_1| = | (L_{\gamma_P}^{-1} (2a,0) )_1| = 2|a| \cosh(\gamma_P),
$$
where $a:=m_\xi^\pi \sinh(\gamma_\xi - \gamma_P)$. As $|\xi^\omega - \tilde{\xi}^\omega|=|\zeta^\omega - \tilde{\zeta}^\omega|$ and the denominator in \eqref{eq:proof hyp 1} is easily seen to be equal to $| a| r_P$ (see the proof of Lemma \ref{convolution:hyp}), the kernel $K_{\omega, \bH^d_m}$ may then be expressed as
$$
K_{\omega, \bH^d_m}=\frac{2(\phi_{m_\xi^\pi} (\xi^\omega) + \phi_{m_\zeta^\pi}(\zeta^\omega)  )}{ (  \phi_{m_\xi^\pi}(\xi^\omega) + \phi_{m_\zeta^\pi} (\zeta^\omega)  )^2-(\xi^\omega + \zeta^\omega)^2 }
$$
after noting that $\cosh(\gamma_P)= (\phi_{m_\xi^\pi}(\xi^\omega) + \phi_{m_\zeta^\pi}(\zeta^\omega) )/r_P$.
\end{remark}


\section{Corollaries}\label{sec:corollaries}


\subsection{Proof of Corollary \ref{cor:sphere1early}}\label{subsec:cor sphere 1}

By \eqref{id:sph} it is clear that the expression
\begin{equation}\label{eq:4wave sph}
g_1(\xi) \bar{g}_2(\xi^\pi + \tilde{\xi}^\perp) g_2(\zeta) \bar{g}_1(\zeta^\pi + \tilde{\zeta}^\perp)
\end{equation}
on its right-hand side is real and positive. The identity \eqref{eq:complex numbers} then yields that \eqref{eq:4wave sph} equals to
$$
\frac{1}{2} \big( |g_1(\xi) g_2(\zeta) |^2 + |g_2(\xi^\pi + \tilde{\xi}^\perp)  g_1(\zeta^\pi + \tilde{\zeta}^\perp)|^2 - |g_1(\xi) g_2(\zeta)  -  g_2(\xi^\pi + \tilde{\xi}^\perp)  g_1(\zeta^\pi + \tilde{\zeta}^\perp)|^2 \big).
$$
The negative term above immediately gives raise to the expression $I_{\pi, \bS^{n-1}}(g_1,g_2)$, whilst the positive terms amount to the same expression over the integral sign, finishing the proof. \qed

\begin{remark}
Observe that the resulting sharp inequality 
\begin{equation}\label{eq:sharp ineq sphere}
\int_{ \pi^\perp} \!\! \Big| (-\Delta_y )^{1/4}  T_{n-2,n}  (\widehat{g_1\ud \sigma^n} \overline{\widehat{g_2\ud \sigma^n}}) (\pi,y) \Big|^2  \! \ud y  \!
 \leq \! \mathbf{C}_{\bS^{n-1}} \!\!  \int_{ (\bS^{n-1})^2 } \!\!\!\!\!\!\! K_{\pi, \bS^{n-1}} (\xi,\zeta) |g_1 ( \xi )|^2 |g_2 (\zeta)|^2  \ud \sigma^n \! (\xi) \, \ud \sigma^n \! (\zeta) 
\end{equation}
obtained from dropping the negative term in \eqref{eq:id sphere minus} may be deduced more directly via a simple application of the Cauchy--Schwarz inequality. Note that \eqref{eq:innermost integral Sigma} is a positive quantity, so in particular equals to its modulus. By the triangle inequality, the left-hand side of \eqref{id:sph} is controlled by
\begin{align}\label{eq:before CS}
\int_{|\xi^\pi| \leq 1} \int_{|\zeta^\pi| \leq 1} \int_{(r_\xi^\pi \bS^1)^2 \times (r_\zeta^\pi \bS^1)^2} \!\!\!\!\!\!\!\! \!\!\!\!   |\xi^\perp -\eta^\perp|^{1/2} |\zeta^\perp-\mu^\perp|^{1/2}| & g_{1,\xi^\pi} (\xi^\perp)| | g_{2,\xi^\pi} (\eta^\perp)| | g_{2,\zeta^\pi}(\zeta^\perp)| | g_{1,\zeta^\pi}(\mu^\perp)| \\
& \, \ud \Sigma^{\perp}_{\xi^\pi, \zeta^\pi}(\xi^\perp, \eta^\perp,\zeta^\perp,\mu^\perp) \, \ud \lambda_\pi (\xi^\pi) \, \ud \lambda_\pi (\zeta^\pi). \notag
\end{align}
Applying the Cauchy--Schwarz inequality with respect to the measure $\ud \Sigma_{\xi^\pi, \zeta^\pi}^\perp \, \ud \lambda_\pi (\xi^\pi) \, \ud \lambda_\pi (\zeta^\pi)$, the above is further controlled by
\begin{equation*}
 \int_{|\xi^\pi| \leq 1} \int_{|\zeta^\pi| \leq 1} \int_{r_\xi^\pi \bS^1 \times r_\zeta^\pi \bS^1} |g_{1,\xi^\pi} (\xi^\perp)|^2 |g_{2,\zeta^\pi}(\zeta^\perp)|^2 \big( h_{\xi} \ud \sigma^\perp_{r_\xi^\pi} \ast^\perp h_{\zeta} \ud \sigma^{\perp}_{r_\zeta^\pi} \big) (\xi^\perp + \zeta^\perp) \, \ud \sigma_{r_{\xi}^{\pi}}^\perp(\xi^\perp)   \, \ud \Sigma_\pi(\xi, \zeta)
\end{equation*}
where $h_{\xi}(\eta^\perp):=|\xi^\perp-\eta^\perp|^{1/2}$ and similarly for $h_{\zeta}$; above $ \ud \Sigma_\pi(\xi, \zeta):= \ud \sigma_{r_{\zeta}^{\pi}}^\perp(\zeta^\perp) \, \ud \lambda_\pi (\xi^\pi) \, \ud \lambda_\pi (\zeta^\pi)$. Evaluation of the innermost convolution as in \eqref{evaluation conv circles proof} yields then the desired inequality \eqref{eq:sharp ineq sphere}. 
\end{remark}


\subsection{Proof of Corollary \ref{cor:sph2}}\label{subsec:cor sph 1}

Given $\pi \in \mathcal{G}_{n-2,n}$, Plancherel's theorem and the relation \eqref{FT k-plane} yields
\begin{align*}
\int_{ \pi^\perp} \Big| (-\Delta_y)^{1/4} T_{n-2,n} h (\pi,y) \Big|^2 \, \ud \lambda_{\pi^\perp} (y) = (2\pi)^{-2} \int_{\pi^\perp} |\xi^\perp| | \widehat{h} ( \xi^\perp)|^2   \ud \lambda_{\pi^\perp} (\xi^\perp).
\end{align*}
Averaging over all $\pi \in \CG_{n-2,n}$, and using \eqref{g=1} and polar coordinates
\begin{align*}
\int_{\CG_{n-2,n}}\int_{ \pi^\perp} |\xi^\perp|^{3-n} | \widehat{h} ( \xi^\perp)|^2 &  |\xi^\perp|^{n-2}  \ud \lambda_{\pi^\perp}(\xi^\perp)  \, \ud \mu_{\CG}( \pi)  \\
& = \int_{\CG_{n-2,n}} \int_0^\infty \int_{ \bS^{n-1} \cap \pi^\perp} r^{3-n} |\widehat{h}(r\omega)|^2 r^{n-2} r \, \ud r \, \ud \sigma^{n,\perp} (\omega) \, \ud \mu_\CG(\pi) \\
& = |\mathcal{G}_{1,n-1}| \int_0^\infty \int_{\bS^{n-1}} r^{3-n} |\widehat{h}(r\omega)|^2 r^{n-1} \, \ud r \, \ud \sigma^n(\omega) \\
& =  |\mathcal{G}_{1,n-1}|  (2\pi)^n \int_{\R^n} ||\nabla|^{\frac{3-n}{2}} h (x) |^2 \, \ud x,
\end{align*}
which completes the proof on taking $h=\widehat{g_1 \ud \sigma^n}\overline{\widehat{g_2\ud \sigma^n}}$. \qed


\subsection{Proof of Corollary \ref{cor:sphST}}

Recall $\xi=\xi^\pi + \xi^\perp$. For $n=3$, $\pi=\langle \omega \rangle$, where $\omega \in \mathcal{G}_{1,3} \simeq \bS^2_+$. Then $\xi^\pi = (\xi \cdot \omega)\, \omega$ and $\xi^\perp = \xi - (\xi \cdot \omega)\,\omega$, so
$$
|\xi^\perp + \zeta^\perp|^2 = |\xi + \zeta|^2 + |(\xi + \zeta) \cdot \omega|^2 - 2 ((\xi + \zeta) \cdot \omega)^2 = |\xi + \zeta|^2 \Big(1- \Big( \frac{(\xi + \zeta)}{|\xi + \zeta|} \cdot \omega \Big)^2\Big).
$$
Noting that $|\mathcal{G}_{1,2}| = \pi$,
$$
K_{\bS^{n-1}}(\xi,\zeta)=
\frac{2}{|\mathcal{G}_{1,2}|} \int_{\bS^2_+} \frac{\ud \sigma^3_+(\omega)}{|\xi^\perp + \zeta^\perp|} = \frac{2 \pi}{\pi |\xi + \zeta|} \int_{-1}^1 \frac{\ud u}{\sqrt{1-u^2}} = \frac{2\pi}{|\xi + \zeta|} .
$$
Thus
$$
\| \widehat{g \ud \sigma^3} \|_{L^4(\R^3)}^4 \leq (2\pi)^4 \int_{\bS^2} \int_{\bS^2} \frac{1}{|\xi + \zeta|} |g(\xi)|^2 |g(\zeta)|^2 \, \ud \sigma^3(\xi) \, \ud \sigma^3(\zeta)
$$
and the desired sharp Stein--Tomas inequality for the sphere follows from the following fact due to Foschi \cite{Foschi}:
\begin{equation}\label{Foschi's eq}
\int_{\bS^2} \int_{\bS^2} \frac{1}{|\xi + \zeta|} |g(\xi)|^2 |g(\zeta)|^2 \, \ud \sigma^3(\xi) \, \ud \sigma^3(\zeta) \leq \| g \|_{L^2(\bS^2)}^4,
\end{equation}
which holds for $g$ antipodally symmetric. The reduction to the antipodally symmetric case may be done as in \cite{Foschi}, using the Cauchy--Schwarz inequality for real numbers 
\begin{equation}\label{eq:CS numbers}
ac+bd \leq \sqrt{a^2+b^2} \sqrt{c^2+d^2}.
\end{equation}
Indeed, note that in the proof of \eqref{eq:sharp ineq sphere} via the Cauchy--Schwarz inequality given in Section \ref{subsec:cor sphere 1}, one may replace $|g_{\xi^\pi}(\xi^\perp)| |g_{\xi^\pi} (\eta^\perp)|$ in the innermost integral in \eqref{eq:before CS} by
$$
\frac{|g_{\xi^\pi}(\xi^\perp)| |g_{\xi^\pi} (\eta^\perp)| + |g_{\xi^\pi}( - \xi^\perp)| |g_{\xi^\pi} ( - \eta^\perp)|}{2},
$$
and using \eqref{eq:CS numbers} this is bounded by $|g_{\xi^\pi}^{\#}(\xi^\perp)| |g_{\xi^\pi}^{\#}(\eta^\perp)|$, where for any function $h$, the function $h^{\#}$ denotes $h^{\#}(\xi):=\sqrt{ (h(\xi) + h(-\xi))/2}$, which is antipodally symmetric. One can argue similarly to replace $|g_{\zeta^\pi}(\zeta^\perp)| |g_{\zeta^\pi} (\mu^\perp)|$ by $|g^{\#}_{\zeta^\pi}(\zeta^\perp)| |g^{\#}_{\zeta^\pi} (\mu^\perp)|$. Thus, the right-hand side in \eqref{eq:sharp ineq sphere} is replaced by
\begin{equation}\label{eq: with hash}
\mathbf{C}_{ \bS^{n-1}} \! \int_{|\xi^\pi| \leq 1}  \! \int_{|\zeta^\pi| \leq 1} \!  \int_{r_\xi^\pi \bS^1 \times r_\zeta^\pi \bS^1} \!\! \frac{2}{|\xi^\perp + \zeta^\perp|}  |g_{\xi^\pi}^{\#} ( \xi^\perp )|^2 |g_{\zeta^\pi}^{\#} (\zeta^\perp)|^2  \, \ud \sigma^1_{r_{\xi^\pi}} (\xi^\perp) \, \ud \sigma^1_{r_{\zeta^\pi}} (\zeta^\perp) \, \ud \lambda_\pi (\xi^\pi) \, \ud \lambda_\pi (\zeta^\pi). 
\end{equation}
One desires, however, to have $g^{\#}$ rather than $g^{\#}_{\xi^\pi}$ and $g_{\zeta^\pi}^{\#}$. By a change of variables, the integrand $4 |g_{\xi^\pi}^{\#} ( \xi^\perp )|^2 |g_{\zeta^\pi}^{\#} (\zeta^\perp)|^2$ may be further replaced by
$$
\big( |g_{\xi^\pi}^{\#} ( \xi^\perp )|^2  + |g_{-\xi^\pi}^{\#} ( \xi^\perp )|^2 \big)  \big( |g_{\zeta^\pi}^{\#} (\zeta^\perp)|^2 + |g_{-\zeta^\pi}^{\#} (\zeta^\perp)|^2 \big),
$$
which equals
$$
|g^{\#}(\xi)|^2 |g^{\#}(\zeta)|^2 + |g^{\#}(\xi)|^2 |g^{\#}(\zeta^\perp - \zeta^\pi)|^2 + |g^{\#}(\xi^\perp - \xi^\pi)|^2 |g^{\#}(\zeta)|^2 + |g^{\#}(\xi^\perp - \xi^\pi)|^2 |g^{\#}(\zeta^\perp - \zeta^\pi)|^2.
$$
A further change of variables in each of the terms allows one to see that \eqref{eq: with hash} equals
$$
\mathbf{C}_{ \bS^{n-1}} \int_{ (\bS^{n-1})^2 } \!\!\! K_{\pi, \bS^{n-1}} (\xi,\zeta) |g^{\#} ( \xi )|^2 |g^{\#} (\zeta)|^2  \, \ud \sigma^n (\xi) \, \ud \sigma^n (\zeta),
$$
as desired for the later application of Foschi's identity \eqref{Foschi's eq} on antipodally symmetric functions. \qed


\subsection{Proof of Corollary \ref{cor:hyp 1}}

This follows the same argument as that of Corollary \ref{cor:sphere1early}. \qed

\subsection{Proof of Corollary \ref{cor:hyp 2}}

The proof follows from the same argument as in  $\S$\ref{subsec:cor sph 1}. Indeed, the elementary argument therein yields the relation
$$
\| (-\Delta)^{\ell/2} f \|_{L^2(\R^d)}^2 = (2\pi)^{-(d-1)} \| \partial_s^{\frac{d-1}{2} + \ell} \CR f \|_{L^2_{\omega, s}(\bS^{d-1}_+, \R)}^2,
$$
from which Corollary \ref{cor:hyp 2} follows from taking $\ell=(2-d)/2$ after averaging over $\omega \in \bS^{d-1}_+$; note that $\omega$ in the Radon transform $\CR$ only runs over $\bS^{d-1}_+ \simeq \mathcal{G}_{d-1,d}$. \qed


\section{The bilinear identity \eqref{PV d} for paraboloids revisited}\label{sec:revisited}

The purpose of this final section is to provide an alternative proof of Theorem \ref{PV thm} via Fourier analysis. The proof follows the same scheme as those of Theorems \ref{thm:sphere1} and \ref{thm:hyp 1} with a little twist, which is available when taking one full derivative in the $s$-variable in the case of paraboloids. 

To see this, let $\mathbb{P}^d_{a}:= \{ (\xi, |\xi|^2+ a) : \xi \in \R^d \}$ denote the paraboloid in $(x,t) \in \R^{d} \times \R$ with tangent plane $t=a$ at its vertex; if $a=0$ we simply denote it by $\bP^d$. Let $\ud \sigma_{\bP^d_a}$ denote the parametrised measure on $\bP^d_a$, which satisfies $\widehat{g \ud \sigma_{\bP^d_a}}(x,t)=Ef(x,t)$ where $E$ is the extension operator associated to $\phi(\xi):=|\xi|^2+a$ and $g$ is the lift of the function $f:\R^d \to \C$ to $\bP^d_a$.

Given $\omega \in \bS^{d-1}_+$ and $\pi= \langle \omega \rangle^\perp \in \mathcal{G}_{d-1,d}$ write, for each $\xi \in \R^d$, $\xi= \xi^\pi + \xi^\omega \omega $, where $\xi^\omega= \xi \cdot \omega $. Given $s \in \R$ and $x \in \pi$,
\begin{align*}
\widehat{g_j \ud \sigma_{\bP^d}} (x+ s \omega, t) &= \int_{\R^d} e^{i (x+s\omega) \cdot \xi + i t |\xi|^2} f_j(\xi) \: \ud \xi \\
& = \int_{\pi} e^{i x \cdot \xi^\pi} \int_{\R}  e^{i s \xi^\omega  + i t |\xi^\pi|^2 + i t |\xi^\omega|^2} f_j ( \xi^\pi + \xi^\omega \omega ) \, \ud \xi^\omega \, \ud \lambda_\pi (\xi^\pi)  \\
& = \int_{\pi} e^{i x \cdot \xi^\pi} \mathcal{F}^2 (g_{j,\xi^\pi} \ud \sigma_{\bP^2_{|\xi^\pi|^2}}) (s,t) \, \ud \lambda_{\pi}(\xi^\pi),
\end{align*}
where $f_{j,\xi^\pi} (\nu) := f_j (\xi^\pi+\nu \omega)$, $\mathcal{F}^2$ denotes the $2$-dimensional Fourier transform and $g_{j,\xi^\pi}$ is the lift of $f_{j,\xi^\pi}$ to $\bP^2_{|\xi^\pi|^2}$. Reasoning as in the proof of Theorem \ref{thm:sphere1},
\begin{align*}
\int_{\R}& \int_{\R}  |\partial_s \mathcal{R} \big( \widehat{g_1 \ud \sigma_{\bP^d} }(\cdot, t) \overline{\widehat{g_2 \ud \sigma_{\bP^d}}}(\cdot, t)\big) (\omega, s) |^2 \, \ud s \, \ud t \\
& \! = \!(2\pi)^{2d}   \! \int_{\pi} \! \int_{\pi} \!  \int_{\R^4} \! |\xi^\omega-\eta^\omega||\zeta^\omega-\mu^\omega| f_{1,\xi^\pi} (\xi^\omega) \bar{f}_{2,\xi^\pi}(\eta^\omega) \bar{f}_{1,\zeta^\pi}(\mu^\omega) f_{2,\zeta^\pi} (\zeta^\omega) \, \ud \Sigma_{\xi^\pi, \zeta^\pi}(\xi^\omega,\eta^\omega,\mu^\omega, \zeta^\omega) 
\end{align*}
where
$$
\ud \Sigma_{\xi^\pi, \zeta^\pi}(\xi^\omega \!,\eta^\omega \!,\mu^\omega\! , \zeta^\omega\! ):=\! \delta (\xi^\omega-\eta^\omega + \zeta^\omega - \mu^\omega)  \delta ( (\xi^\omega)^2-(\eta^\omega)^2 + (\zeta^\omega)^2 - (\mu^\omega)^2)  \ud \xi^\omega  \ud \eta^\omega  \ud \mu^\omega \ \ud \zeta^\omega  \ud \lambda_{\pi} (\xi^\pi) \ud \lambda_{\pi} (\zeta^\pi).
$$
Arguing similarly,
\begin{align*}
J_\omega (\widehat{g_1 \ud \sigma_{\bP^d}}, \widehat{g_2 \ud \sigma_{\bP^d}})\! = \! (2\pi)^{2d} \!   \int_{\pi} \! \int_{\pi} \!  \int_{\R^4} \!  (\zeta^\omega \mu^\omega- \zeta^\omega \eta^\omega - \xi^\omega \mu^\omega + \xi^\omega \eta^\omega) f_{1,\xi^\pi} (\xi^\omega) \bar{f}_{2,\xi^\pi}(\eta^\omega) \bar{f}_{1,\zeta^\pi}(\mu^\omega) f_{2,\zeta^\pi} (\zeta^\omega) 
\end{align*}
with respect to the measure $\ud \Sigma_{\xi^\pi, \zeta^\pi}(\xi^\omega,\eta^\omega,\mu^\omega, \zeta^\omega)$, where $J_\omega(G_1,G_2)$ is the bilinearisation of $J_\omega(u)$; namely the integrand is replaced by
\begin{align*}
G_1 & (x + s\omega,t)  \partial_s G_2 (y+ s\omega,t) \big( \partial_s \bar{G_1} (y+ s\omega,t) \bar{G_2}(x+ s\omega,t) -  \bar{G_1} (y+ s\omega,t) \partial_s \bar{G_2}(x+ s\omega,t)  \big) \\
&  - G_2 (y+ s\omega,t) \partial_s G_1 (x+ s\omega,t) \big( \bar{G_2} (x+ s\omega,t) \partial_s \bar{G_1} (y+ s\omega,t) - \partial_s \bar{G_2} (x+ s\omega,t)  \bar{G_1} (y+ s\omega,t) \big) .
\end{align*}
Noting that
\begin{equation}\label{eq:algebraic}
 |\xi^\omega-\eta^\omega||\zeta^\omega-\mu^\omega| + (\zeta^\omega \mu^\omega- \zeta^\omega \eta^\omega - \xi^\omega \mu^\omega + \xi^\omega \eta^\omega) = |\xi^\omega - \mu^{\omega}|^2
\end{equation}
if $(\xi^\omega, \eta^\omega, \mu^\omega, \zeta^\omega) \in \supp \: ( \ud \Sigma_{\xi^\pi, \zeta^\pi})$, one can combine the two terms above to obtain
\begin{align}
\label{eq:adding both terms PV}
\int_{\R} & \int_{\R}  |\partial_s \mathcal{R} \big(  \widehat{g_1 \ud \sigma_{\bP^d}}(\cdot, t) \overline{\widehat{g_2 \ud \sigma_{\bP^d}}}(\cdot, t)\big) (s,\omega) |^2 \, \ud s \, \ud t  + J_\omega(\widehat{g_1 \ud \sigma_{\bP^d}},\widehat{g_2 \ud \sigma_{\bP^d}})  \\
& =  (2\pi)^{2d} \!  \int_{\pi} \! \int_{\pi} \!  \int_{\R^4}  \!\! |\xi^\omega-\mu^\omega|^2 f_{1,\xi^\pi} (\xi^\omega) \bar{f}_{2,\xi^\pi}(\eta^\omega) \bar{f}_{1,\zeta^\pi}(\mu^\omega) f_{2,\zeta^\perp} (\zeta^\omega) \, \ud \Sigma_{\xi^\pi, \zeta^\pi}(\xi^\omega,\eta^\omega,\mu^\omega, \zeta^\omega). \notag
\end{align}
For fixed $\xi^\omega$ and $\mu^\omega$, the only solution for the equations in the $\delta$ function is $\eta^\omega=\xi^\omega$ and $\zeta^\omega=\mu^\omega$. Thus, the right-hand side above equals
$$
\frac{ (2\pi)^{2d}}{2}  \int_{\xi^\pi} \int_{\zeta^\pi}  \int_{\R^2}  |\xi^\omega-\mu^\omega| f_{1,\xi^\pi} (\xi^\omega) \bar{f}_{2,\xi^\pi}(\xi^\omega) \bar{f}_{1,\zeta^\pi}(\mu^\omega) f_{2,\zeta^\pi} (\mu^\omega) \, \ud \xi^\omega \,  \ud \mu^\omega \, \ud \lambda_{\pi} (\xi^\pi) \, \ud \lambda_{\pi} (\zeta^\pi)
$$
and if $f_1=f_2$,
$$
\frac{ (2\pi)^{2d}}{2}  \int_{\pi} \int_{\pi}  \int_{\R^2}  |\xi^\omega-\mu^\omega| |f_{\xi^\pi} (\xi^\omega)|^2 |f_{\zeta^\pi} (\mu^\omega)|^2 \, \ud \xi^\omega \,  \ud \mu^\omega \, \ud \lambda_{\pi} (\xi^\pi) \, \ud \lambda_{\pi} (\zeta^\pi)
$$
which, of course, is
$$
\frac{ (2\pi)^{2d}}{2} \int_{\R^d} \int_{\R^d}  |(\xi-\eta) \cdot \omega| |f(\xi)|^2 |f(\eta)|^2 \, \ud \xi \,  \ud \eta.
$$
In the language of the Schrödinger equation, $u=E\widecheck{u_0}$, so the right-hand side is
$$
\frac{ \pi}{(2\pi)^{d+1}} \int_{\R^d} \int_{\R^d}  |(\xi-\eta) \cdot \omega| |\widehat{u_0}(\xi)|^2 |\widehat{u_0}(\eta)|^2 \, \ud \xi \,  \ud \eta
$$
and one obtains the desired identity \eqref{PV d}, finishing the proof of Theorem \ref{PV thm}. \qed

\begin{remark}
Averaging over all $\omega \in \bS_+^{d-1}$ after dropping the term $J_{\omega}(u)$ from the obtained identity and noting that
$$
\int_{\bS^{d-1}} |(\xi-\eta) \cdot \omega| \, \ud \sigma^n(\omega)=2 |\xi-\eta| \int_0^1 u (1-u^2)^{\frac{d-3}{2}} \, \ud u = \frac{2 |\xi-\eta| \pi^{\frac{d-1}{2}}  }{\Gamma((d+1)/2)},
$$
one has
\begin{equation}\label{eq:PV ineq}
\| (-\Delta_x)^{\frac{3-d}{4}} ( |u|^2) \|^2_{L^2_{x,t}(\R^d \times \R)} \leq (2\pi)^{1-d} \frac{\pi}{(2\pi)^{d+1}} \frac{\pi^{\frac{d-1}{2}}}{\Gamma((d+1)/2)}
\int_{\R^d} \int_{\R^d}  |\xi-\eta| \widehat{u_0}(\xi)|^2 |\widehat{u_0}(\eta)|^2 \, \ud \xi \,  \ud \eta
\end{equation}
and the constant simplifies as $\textbf{PV}(d):=\frac{2^{-3d} \pi^{\frac{1-5d}{2}}}{\Gamma(\frac{d+1}{2})}$; this inequality was also obtained in \cite{BBJP} in a more direct way.
\end{remark}

Finally, it is noted that  the honest analogue of Theorems \ref{thm:sphere1} and \ref{thm:hyp 1} in the context of paraboloids is given by the following bilinear identity.
\begin{theorem}\label{thm:honest paraboloid}
Let $d \geq 2$ and $\omega \in \bS^{d-1}_+$. Then
\begin{align}
\int_{\R}  \int_{\R}  & |\partial_s^{1/2}  \mathcal{R} \big( \widehat{g_1 \ud \sigma_{\bP^d}}(\cdot, t) \overline{\widehat{g_2 \ud \sigma_{\bP^d}}}(\cdot, t)\big) (s,\omega) |^2 \, \ud s \, \ud t \label{eq:paraboloids bilinear identity} \\ \notag
&= \!  \frac{(2\pi)^{2d}}{2} \!  \int_{\pi} \! \int_{\pi} \!  \int_{\R^2} \! f_{1,\xi^\pi} (\xi^\omega) \bar{f}_{2,\xi^\pi}(\zeta^\omega) \bar{f}_{1,\zeta^\pi}(\xi^\omega) f_{2,\zeta^\pi} (\zeta^\omega)  \, \ud \xi^\omega  \, \ud \zeta^\omega \, \ud \lambda_{\pi} (\xi^\pi) \, \ud \lambda_{\pi} (\zeta^\pi).
\end{align}
\end{theorem}
The proof of Theorem \ref{thm:honest paraboloid} is a minor variant of the one for Theorem \ref{PV thm} exposed above: the main difference is that here one solves the equations in the $\delta$ functions in terms of $\xi^\omega$ and $\zeta^\omega$; the solution in terms of $\xi^\omega$ and $\mu^\omega$ is now degenerate in terms of the weight $|\xi^\omega-\eta^\omega|^{1/2} |\zeta^\omega-\mu^\omega|^{1/2}$, which vanishes in this case. Note that, in \eqref{eq:adding both terms PV}, the fact of taking one full derivative with respect to $s$ and adding the term $J_\omega (\widehat{g_1 \ud \sigma_{\bP^d}}, \widehat{g_2 \ud \sigma_{\bP^d}})$ had the effect of replacing the weight $|\xi^\omega-\eta^\omega|  |\zeta^\omega-\mu^\omega|$ by $|\xi^\omega - \mu^\omega|^2$ thanks to the algebraic identity \eqref{eq:algebraic}, allowing one to solve in those variables.

Corollaries in the spirit of those obtained for spheres and hyperboloids can also be deduced from Theorem \ref{thm:honest paraboloid}. In particular, the identity for complex numbers \eqref{eq:complex numbers} allows one to rewrite \eqref{eq:paraboloids bilinear identity} as the following.
\begin{corollary} Let $d \geq 2$ and $\omega \in \mathbb{S}^{d-1}_+$. Then
\begin{align*}
\int_{\R} \int_{\R}  |\partial_s^{1/2}  \mathcal{R}  \big( & \widehat{g_1 \ud \sigma_{\bP^d}}(\cdot, t) \overline{\widehat{g_2 \ud \sigma_{\bP^d}}}(\cdot, t)\big) (s,\omega) |^2 \, \ud s \, \ud t  \\ 
&= \frac{ (2\pi)^{2d}}{2}   \int_{\pi} \int_{\pi}  \int_{\R^2} | f_{1,\xi^\pi} (\xi^\omega)|^2 |f_{2,\zeta^\pi}(\zeta^\omega)|^2   \, \ud \xi^\omega  \, \ud \zeta^\omega \, \ud \lambda_{\pi} (\xi^\pi) \, \ud \lambda_{\pi} (\zeta^\pi) - I_{\omega}(f_1,f_2)
\end{align*}
where
$$
I_\omega(f_1,f_2):= \frac{(2\pi)^{2d}}{4} \int_{\pi} \int_{\pi}  \int_{\R^2} |f_{1,\xi^\pi} (\xi^\omega) f_{2,\zeta^\pi}(\zeta^\omega)  - f_{1,\zeta^\pi}(\xi^\omega) f_{2,\xi^\pi} (\zeta^\omega)|^2
  \, \ud \xi^\omega  \, \ud \zeta^\omega \, \ud \lambda_{\pi} (\xi^\pi) \, \ud \lambda_{\pi} (\zeta^\pi).
$$
\end{corollary}
Note that, unlike $J_\omega(f)$, the term $I_\omega(f,f)$ does not have an obvious closed expression in terms of physical variables. Setting $f_1=f_2$ and averaging over all $\omega \in \bS^{d-1}_+$ after dropping $I_\omega(f,f)$ one obtains 
$$
\| (-\Delta_x)^{\frac{2-d}{4}} ( |u|^2) \|^2_{L^2_{x,t}(\R^d \times \R)} \leq (2\pi)^{1-d} \frac{(2\pi)^{2d}}{2} \frac{|\bS^{d-1}|}{2}  \| \widecheck{u_0} \|_{L^2(\R^d)}^4 = \frac{2^{-d} \pi^{\frac{2-d}{2}}}{\Gamma(d/2)} \| u_0 \|_{L^2(\R^d)}^4,
$$
which is the Ozawa--Tsutsumi estimate \eqref{OT d}; note that for $d=2$ this amounts to the $L^4(\R^{2+1})$ Strichartz estimate. The interested reader should look at the work of Bennett, Bez, Jeavons and Pattakos \cite{BBJP} for a unified treatment of the Ozawa--Tsutsumi estimates \eqref{OT d}, the inequalities deduced from \eqref{eq:PV ineq}, and a more general case with an arbitrary number of derivatives on the left-hand-side of such inequalities.


\bibliographystyle{abbrv} 

\bibliography{reference}




%
%
%
%
%

\end{document}